\documentclass[11pt]{amsart}
\usepackage{amssymb,amsthm,amsmath}
\usepackage[margin=1in]{geometry}
\usepackage{enumitem} 
\usepackage{tikz-cd}
\usepackage[all,cmtip]{xy}
\setlist[enumerate]{label={(\roman*)}}
\numberwithin{equation}{section}
\newtheorem{theorem}{Theorem}[section]

\newtheorem{prop}[theorem]{Proposition}
\newtheorem{corollary}[theorem]{Corollary}
\newtheorem{hypothesis}[theorem]{Hypothesis}
\newtheorem{lemma}[theorem]{Lemma}
\theoremstyle{definition}
\newtheorem{definition}[theorem]{Definition}

\newtheorem{remark}[theorem]{Remark}               

\theoremstyle{plain}

\newtheorem{thmalph}{Theorem}

\newcommand{\surj}{\twoheadrightarrow}
\newcommand{\inj}{\hookrightarrow}
\newcommand{\lra}{\longrightarrow}
\newcommand{\ol}{\overline}

\newcommand{\C}{\mathbb{C}}
\newcommand{\Q}{\mathbb{Q}}

\newcommand{\N}{\mathbb{N}}
\newcommand{\Z}{\mathbb{Z}}
\newcommand{\Zp}{{\mathbb{Z}_p}}

\newcommand{\gal}{\operatorname{Gal}}
\newcommand{\ind}{\operatorname{ind}}
\newcommand{\loc}{\operatorname{loc}}
\newcommand{\End}{\operatorname{End}}
\newcommand{\img}{\operatorname{image}}
\newcommand{\Sel}{\operatorname{Sel}}
\newcommand{\stub}{\operatorname{Stub}}
\newcommand{\Stub}{\operatorname{Stub}}
\newcommand{\len}{\operatorname{length}}
\newcommand{\frob}{\operatorname{Frob}}
\newcommand{\ord}{\operatorname{ord}}
\newcommand{\ES}{\operatorname{ES}}

\newcommand{\fa}{\mathfrak{a}}
\newcommand{\fb}{\mathfrak{b}}
\newcommand{\mm}{\mathfrak{m}}
\newcommand{\nn}{\mathfrak{n}}
\newcommand{\fl}{\mathfrak{l}}
\newcommand{\fq}{\mathfrak{q}}

\newcommand{\fP}{\mathfrak{P}}
\newcommand{\OO}{\mathcal{O}}
\newcommand{\cL}{\mathcal{L}}
\newcommand{\cF}{\mathcal{F}}
\newcommand{\cN}{\mathcal{N}}
\newcommand{\cX}{\mathcal{X}}

\newcommand{\cS}{\mathcal{S}}

\begin{document}

\title[Bipartite Euler systems for certain Galois representations]{Bipartite Euler Systems for certain Galois representations}
\author{Chandrakant Aribam}
\author{Pronay Kumar Karmakar}
\date{}

\begin{abstract} 
Let $E/\Q$ be an elliptic curve with ordinary reduction at a prime $p$, and let $K$ be an imaginary quadratic field. The anticyclotomic Iwasawa main conjecture, depending upon the sign of the functional equation of $L(E/K,s)$, predicts the behavior of Selmer group of $E/\Q$ along the anticyclotomic tower of $K$. Some of the crucial ideas of Bertolini and Darmon on this conjecture have been abstracted by Howard into an axiomatic set-up through a notion of Bipartite Euler systems, assuming that $E[p]$ is an irreducible representation of $G_{K}$. We generalize this work by assuming only $(E[p])^{G_K}=0$. 
We use the results of Howard, Nekov\'a\v{r} and Castella \emph{et al}., along with those of Mazur and Rubin on Kolyvagin systems to show one divisibility of the anticyclotomic main conjecture, for both the signs. The other divisibility can be reduced to proving the nonvanishing of sufficiently many $p$-adic $L$-functions attached to a family of congruent modular forms.
\end{abstract}

\maketitle 

\tableofcontents

\section{Introduction}
Let $E/\Q$ be an elliptic curve defined over $\Q$ of conductor $N=N^+N^-$ and $K$ be an imaginary quadratic field extension of $\Q$, where $N^+$ ( resp. $N^-$) is divisible only by primes that split ( resp. remain inert) in $K$. Let $D_K$ be the discriminant of $K$ such that $(D_K,N)=1$. Let $L(E/K,s)$ be the Hasse-Weil $L$-function of $E$ and $\epsilon$ be the quadratic character attached to the extension $K/\Q$. It is known from the work of Gross and Zagier that the sign of the functional equation $L(E/K,s)$ is $-\epsilon(N^-)$. Now let $p$ be a rational prime such that $E$ has good ordinary reduction at $p$, $p\nmid6D_KN$ and $K^{ac}$ be the anticyclotomic $\Z_p$ extension of $K$ with $\Gamma^{-}=\gal(K^{ac}/K)$. Let $\Lambda=\Z_p[[\Gamma^{-}]]$ be the Iwasawa algebra of $\Gamma^{-}$. The \emph{anticyclotomic main conjecture for $E$} predicts, when $\epsilon(N^-)=1$ ( known as the indefinite case), that the $p^{\infty}$-Selmer group $\Sel_{p^{\infty}}(K^{ac},E[p^\infty])$ has $\Lambda$-corank one and the characteristic ideal of its $\Lambda$-cotorsion submodule can be expressed in terms of Heegner points, and when $\epsilon(N^-)=-1$ ( known as the definite case),  $\Sel_{p^{\infty}}(K^{ac},E[p^\infty])$ is predicted to be $\Lambda$-cotorsion with characteristic ideal related to the $p$-adic $L$-function. 

Beginning with the monumental and fundamental work of Bertolini and Darmon on this conjecture \cite{bd} in the definite case, much is known when the residual representation $\ol\rho_E:G_\Q \lra GL_2(\mathbb{F}_p)$ of $E$ is \emph{irreducible}, with surjective image, and the modular form $f$ attached to $E$ is $p$-isolated.  
Some aspects of the indefinite case are also treated in \cite{ber}. When the residual representation $\ol\rho_E$ is \emph{reducible}, some cases have been proved by Castella \emph{et al}. in \cite{cgls} in the indefinite case using the theory of Kolyvagin systems, under an assumption of the Heegner hypothesis where $N^-$ is required to be equal to 1. Each of these works depend on the work of many authors, but they all have been inspired by the work of Bertolini and Darmon \cite{bd}. 

Adapting the work of Mazur and Rubin (\cite{mr}) on Kolyvagin systems, the argument of Bertolini and Darmon has been axiomatized by Howard in \cite{hbes}, and it allows one to treat the definite and the indefinite cases simultaneously. Using this axiomatic set-up, in fact, the definite and indefinite cases are treated simultaneously, when $E[p]$ is irreducible, in \cite{cbc}. Howard's axiomatization is done through a notion of \emph{bipartite} Euler system, which we briefly explain.  Let $f$ be the modular form attached to $E/\Q$ (by the work of Wiles, Taylor-Wiles, Diamond). For a choice of positive integer $k$ one can define the set of $k$-admissible primes $\mathcal{L}_k$, all of which are inert in $K$, with the property that for any $\nn \in \mathcal{N}_k$ ( the set of squarefree products of primes in $\mathcal{L}_k$) there is a modular form $f_\nn$ of level $\nn N^-$ which is congruent to $f$ modulo $p^k$. One then considers a graph whose vertices are the elements in $\mathcal{N}_k$ with edges connecting $\nn$ to $\nn\fl$ for each $\fl \in \mathcal{L}_k$ coprime
to $\nn \in \mathcal{N}_k$. A vertex $\nn$, generated by $n\in\N$, is said to be \emph{definite} or \emph{indefinite} depending on whether  $\epsilon(nN^-)$ is $-1$ or $1$ respectively. In this way, we get a bipartition of the graph: every edge connects a definite vertex to an indefinite vertex.
In the axiomatic set-up in \cite{hbes}, Howard, as in \cite{mr}, considers, a rank two module $T$ defined over a principal Artinian local ring $A$, with maximal ideal $\mm$, on which there is a continuous action of $G_{K}$, and a perfect $G_K$-equivariant alternating pairing into $A(1)$. In this set-up, an admissible set of primes $\cL$ of $K$ is defined which consist of primes $\fl$ of $K$ which are inert and the Frobenius acts on $T$ with eigenvalues $N_{K/\Q}(\fl)$ and $1$.
The set of squarefree products of $\cL$ is denoted by $\cN$. Then to each $\nn$, a Selmer group $\Sel_{\cF(\nn)}$ is defined as a subspace of $H^1(K,T)$ by imposing the unramified conditions at primes $\fq\nmid\nn$ and an ordinary condition at $\fl\mid\nn$.

Further,
following \cite{mr}, Howard defines the sheaf of Euler systems by attaching to an even vertex $\nn$ the module $A$, to an odd vertex $\Sel_{\cF(\nn)}$, and to the edge $e(\nn,\nn\fl)$ an appropriate submodule of $H^1(K_\fl,T)$.
For each $k\in\N$, consider the $\Z_p/p^k{\Z_p}$-module $E[p^k]$. Then at an indefinite vertex the modular form $f_\nn$ allows one to define a cohomology class $\kappa_\nn \in \varprojlim_{m} H^1(K_m,E[p^k])$, which arises as the Kummer image of Heegner points on the abelian variety attached to $f_\nn$. At the definite vertex one can attach to $f_\nn$ a $p$-adic L-function $\lambda_\nn \in \Lambda/p^k{\Lambda}$. There are reciprocity laws relating the elements at any two adjacent vertices, and the families
\begin{equation*}
    \{\kappa_\nn | \nn \in \mathcal{N}_k, \nn \mbox{ indefinite } \}\quad  \{\lambda_\nn | \nn \in \mathcal{N}_k, \nn \mbox{ definite } \}
\end{equation*}
forms a \emph{bipartite Euler system}. 
Howard shows how  a sufficiently non-trivial bipartite Euler system can be used to bound the lengths of $\Sel_{\cF(\nn)}$, and proves a rigidity theorem which 
gives a uniform estimate for those bounds for all the vertices.

In this manuscript, we extend the work of Howard in  \cite{hbes} where it is assumed that $\ol\rho_E$ is irreducible to the situation where $(E[p])^{G_K}=0$. Note that $(E[p])^{G_K}=0$ is automatically satisfied when $E[p]$ is irreducible representation of $G_K$.
In trying to do this, we are faced with a problem of not having a Chebotarev Density Theorem. In the \emph{irreducible} case, this is crucial to 
getting the bounds on Selmer groups from an Euler system.  
In the case when $(E[p])^{G_K}=0$, we use results of \cite{cgls}, which further relies on the work of Nekov\'a\v{r} \cite{nekovar}, that allows us to have a result similar to the Chebotarev Density Theorem, but for the twist $T=E[p^k]\otimes\alpha$ for some non-trivial character $\alpha$ of $\Gamma^-$ (see \S \ref{section-key} below). We write $A=\Z_p/p^k\Z_p$.
\begin{thmalph}[Lemma \ref{key-lemma}] Let $\len(A)>{\varepsilon_0}=r(C_1+C_2+C_{\alpha})$. Then for any $\nn\in \mathcal{N}$ and any cyclic free $A$-module $C \subset \Sel_{\mathcal{F}(\nn)}$, there are infinitely many $\fl \in \mathcal{L}$ such that $\loc_\fl$ takes $C$ isomorphically onto $H^1_{unr}(K_\fl,T)$.
\end{thmalph}

Taking $\varepsilon_0=0$ in the above theorem, we can recover the \cite[Lemma 2.3.3]{hbes}. 
However, relaxing the irreducibility condition ends up in introducing an error in the bounds for lengths of Selmer groups ( see Theorem \ref{bound}).
\begin{thmalph}[Theorem \ref{bound}]
Let $\len(A)>\varepsilon_0$. 
For any free Euler system of odd type for $(T,\mathcal{F},\mathcal{L})$, ${\pi}^\varepsilon\lambda_\nn \in \stub_\nn$ for every $\nn \in \mathcal{N}^{even}$, and ${\pi}^\varepsilon\kappa_\nn \in \stub_\nn$ for every  $\nn \in \mathcal{N}^{odd}$. Equivalently, in terms of the $A$-module $M_\nn$ in the decomposition 
$\Sel_{\cF(\nn)}\cong A^{e(\nn)}\oplus M_\nn\oplus M_\nn$, we have
\begin{equation*}
      \len(M_\nn) \leq \begin{cases}
          \ind(\lambda_\nn,A)+\varepsilon &\mbox{ if }\nn \in \mathcal{N}^{even}\\
          \ind(\kappa_\nn,\Sel_{\mathcal{F}(\nn)}(K,T))+\varepsilon &\mbox{ if } \nn \in \mathcal{N}^{odd}.
      \end{cases}
\end{equation*}
\end{thmalph}
The second problem that we faced is in trying to get an analogous rigidity theorem, as in \cite[Theorem 2.5.1]{hbes} ( see Theorem \ref{Rigidity}) assuming only $(E[p])^{G_K}=0$, more precisely, in trying to characterize the core vertices (see Def 2.4.2 of \cite{hbes}). However, the notion of a core vertex, as in \cite{mr} or \cite{hbes}, as a vertex whose Stub module is $A$ (more precisely, $\nn$ is core if $\len(M_\nn)=0$), is not suitable as we get an error term $\varepsilon_0$ ( in the \emph{irreducible} case this error term is zero). Therefore, we look at those vertices $\nn$, such that $\len(M_\nn)$ is minimal. Calling this length $u$, the Stub module is isomorphic to $\mm^uA$, and we call these vertices \emph{absolute core vertices} ( see Def \ref{def-core}). We show that this is well defined by characterizing these vertices in terms of vanishing under localizations and we call them \emph{universally trivial} vertices ( see Def \ref{universally-trivial}). Using these absolute core vertices we show that the graph we have considered is path connected ( see Proposition \ref{path}). As a result we get our desired Rigidity Theorem in Theorem \ref{Rigidity}.

For each $\nn\in\cN_k$, consider the ordinary conditions at primes $\fl\mid\nn$ which allows us to define the Selmer groups $\cS_{\nn}(K^{ac},E[p^k]) \subset \varprojlim_{m} H^1(K_m,E[p^k])$ for $1\leq k\leq\infty$ (see Def \ref{compact-selmer} below). Writing $\cS=\cS_1(K^{ac},T_p)$ and $X$ for the Pontryagin dual of $\Sel(K^{ac},E[p^\infty])$, and assuming that the families of $\kappa_\nn$ and $\lambda_\nn$ are compatible as $k$ varies, we have distinguished elements $\lambda^\infty\in\Lambda$ and $\kappa^\infty\in \cS$.

We then have the \emph{Main theorem} which gives us a bound on the Selmer groups and an equality condition. As a consequence it  can be applied to the anticyclotomic Iwasawa main conjecture for certain \emph{reducible} Galois representations. 
\begin{thmalph}[Theorem \ref{Main Theorem}]
Let the distinguished elements $\lambda^\infty, \kappa^\infty$ be nonzero, and let $X_{{tor}}$ denote the $\Lambda$-torsion submodule of $X$. Then, we have,
\begin{enumerate}
\item The rank formulas:
\begin{equation*}
    \operatorname{rank}_{\Lambda}\cS = \operatorname{rank}_{\Lambda}X = 
    \begin{cases} 
    0 \quad \mbox{if} \quad \epsilon(N^-)=-1\\
    1 \quad \mbox{if} \quad \epsilon(N^-)=1.
    \end{cases}
\end{equation*}
\item For any height one prime $\fP$ of $\Lambda$, we have
\begin{equation*}
    \ord_{\fP}(\operatorname{char}{(X_{{tor}})}) \leq 2. 
    \begin{cases} \ord_{\fP}(\lambda^{\infty}) & \mbox{if} \quad \epsilon(N^-)=-1\\
    \ord_{\fP}(\operatorname{char}{(\cS/\Lambda{\kappa}^{\infty})}) & \mbox{if} \quad \epsilon(N^-)=1.
    \end{cases}
\end{equation*}
\item The inequality (ii) is an equality if the following condition is satisfied: there exists a positive integer $s$ such that for all $t \geq s$ the set
\begin{equation*}
    \{\lambda_{\nn} \in \Lambda/p^t{\Lambda} | \nn \in \mathcal{N}^{definite}_t \}
\end{equation*}
contains an element which is nonzero in $\Lambda/(\fP, p^s)$.
\end{enumerate}
\end{thmalph}
\section{Selmer Groups}\label{selmer} 
In this section, we recall some basic facts about Selmer groups over Artinian rings, as in Mazur-Rubin \cite{mr} and  Howard \cite{hbes}.
The decomposition of the local cohomology groups and the definition of the modified Selmer groups are already present in these references. Much of the notations introduced in \cite{hbes} are also retained.
The difference is in Hypothesis \ref{key-hypothesis}, where we relax the hypothesis that the residual representation is \emph{irreducible}.

Throughout, we fix an algebraic closure $\ol\Q$, a prime $p>2$, and embedding $\iota_p:\ol\Q\inj \ol\Q_p$ and $\iota_\infty:\ol\Q\inj\C$. 
Let $G_\Q=\gal(\ol\Q/\Q)$ and $G_K=\gal(\ol\Q/K)$, and for each place $w$ of $K$ let $G_w\subset G_K$ denote the decomposition group, and $I_w$ its inertia subgroup. We denote the arithmetic Frobenius at the place $w$ by $\frob_w\in G_w/I_w$. For the prime $v\mid p$, let $G_v$ denote the decomposition group which is identified with $\gal(\ol\Q_p/\Q_p)$ via the embedding $\iota_p$. For any place
$w$ of $K$, the maximal unramified extension of $K_w$ is denoted
by $K_w^{unr}$.

Let $K^{ac}$ denote the \emph{anticyclotomic} $\Z_p$-extension of $K$, and $\Gamma^-=\gal(K^{ac}/K)$. By $\Lambda=\Z_p[[\Gamma^-]]$ we
denote the anticyclotomic Iwasawa algebra. By fixing a topological generator $\gamma\in\Gamma^-$, we have a non-canonical 
isomorphism $\Lambda\cong\Z_p[[T^-]]$ via $T^-=\gamma-1$.

Let $A$ be a \emph{principal Artinian local} ring with maximal ideal $\mm=\pi A$ and residue characteristic $p\geq 3$. It is to be noted that the quotient of a discrete valuation ring by some $k$-th power of its maximal ideal is such a ring; conversely, it is not difficult to show that every principal local Artinian ring is a quotient of a discrete valuation ring.
Let $T$ be a free $A$-module of rank two equipped with a continuous
(for the discrete topology) action of $G_K=\gal(\ol{K}/K)$ for some number field
$K$. We assume that $T$ admits a perfect, $G_K$-equivariant, alternating $A(1)$-valued
pairing. 
\begin{definition}
For each prime $w$ of $K$ where $T$ is unramified, and for the maximal unramified extension $K^{unr}$ of $K$ the \emph{unramified cohomology} is defined by
\begin{equation*}
    H^1_{unr}(K_w,T):=ker[H^1(K_w,T)\lra H^1(K_w^{unr},T)].
\end{equation*}
\end{definition}
\begin{definition} 
\begin{enumerate}
\item[(i)] 
A \emph{Selmer structure} $(\mathcal{F},{\Sigma}_{\mathcal{F}})$ on $T$ is a finite set of places ${\Sigma}_{\mathcal{F}}$ of $K$ containing the archimedean places, the primes at which $T$ is ramified, all the primes above $p$; and, for every place $w$ of $K$, a choice of submodule 
 $H^1_{\mathcal{F}}(K_w, T){\subset}H^1(K_w,T)$ such that $H^1_{\mathcal{F}}(K_w, T)$=$H^1_{unr}(K_w,T)$ for all $w{\notin}{\Sigma}_{\mathcal{F}}$. 
\item[(ii)]  
A Selmer structure $\mathcal{F}$ is \emph{self-dual} if the submodule $H^1_{\mathcal{F}}(K_w,T)$ is maximal isotropic under the (symmetric) local Tate pairing
 \begin{equation*}
H^1(K_w,T)\times H^1(K_w,T){\lra}H^2(K_w,A(1)){\cong}A
 \end{equation*}
for every finite place $w{\in}{\Sigma}_{\mathcal{F}}$.
\item[(iii)] 
Under the embedding from $\ol{K}$ to $\ol K_w$ for every place $w$,
which is fixed, we consider the \emph{localization} maps which is 
given by \emph{restriction}:
\begin{equation*}
    \loc_w:H^1(K,T){\lra}H^1(K_w,T).
\end{equation*}  
\item[(iv)] Define the \emph{Selmer module or group}
$\Sel_{\mathcal{F}}=\Sel_{\mathcal{F}}(K,T)$ associated to $\mathcal{F}$ by the exactness of
\begin{equation*}
  0{\lra}\Sel_{\mathcal{F}}{\lra}H^1(K,T)\stackrel{\oplus_w\loc_w}{\lra}{\bigoplus}H^1(K_w,T)/H^1_{\mathcal{F}}(K_w,T),
\end{equation*}
where we still use the notation $\loc_w$ for the localization maps composed with the quotient maps.
\end{enumerate}
\end{definition}
\begin{remark}\label{Note}
\begin{enumerate}
\item A Selmer structure $(\cF,\Sigma_{\cF})$ on $T$ induces a Selmer structure in  a natural way on a $G_K$-submodule (resp. quotient) $S$ of $T$. If $S$ is a submodule of $T$, then the preimage of $H^1_{\cF}(K_w,T)$ under $H^1(K_w,S){\lra}H^1(K_w,T)$ (resp. the image of $H^1_{\cF}(K_w,T)$ under $H^1(K_w,T){\lra}H^1(K_w,S)$) for every place $w$ of $K$ defines a Selmer structure. By \cite[Lemma 1.1.9]{mr}, $H^1_{\cF}(K_w, S)=H^1_{unr}(K_w,S)$ for every $w{\notin}{\Sigma}_{\cF}$, and so this is well-defined. 
\item For $p{\neq}2$, we have $H^1(K_w,T)=0$ for all $w$ archimedean. 
\end{enumerate}
\end{remark}

\begin{definition}
A set of primes $\mathcal{L}$ of $K$ which is disjoint from ${\Sigma}_{\mathcal{F}}$ and satisfies
\begin{enumerate}
    \item for all $\fl{\in}{\mathcal{L}}, N(\fl){\not\equiv}1\pmod p$,
    \item for all $\fl{\in}{\mathcal{L}}$, the Frobenius $\frob_\fl$ acts on $T$ with eigenvalues $N(\fl)$ and $1$,
\end{enumerate}
is called an \emph{admissible set of primes} and is denoted by $\cL$.
\end{definition}
\begin{definition}
For each $\fl\in\cL$, 
$T{\cong}A{\oplus}A(1)$ as a $G_{K_\fl}$-module, and that the decomposition is unique. Then 
the \emph{ordinary cohomology} is defined by 
\begin{equation*}
H^1_{ord}(K_\fl,T):=\operatorname{image} [H^1(K_\fl,A(1)){\lra}H^1(K_\fl,T)].    
\end{equation*}
\end{definition}
By the local Tate duality, $H^1_{\mathcal{F}}(K_\fl,T)=H^1_{unr}(K_\fl,T)$ is maximal isotropic
for all $\fl\notin\Sigma_{\mathcal F}$.
 For each $\fl\in\cL$, by Lemma \cite[Lemma 2.2.1]{hbes}, these cohomology groups appear as direct summands of $G_{K_\fl}$-modules of the local cohomology groups:
\begin{equation}\label{decompo}
H^1(K_\fl,T){\cong}H^1_{unr}(K_\fl,T){\oplus}H^1_{ord}(K_\fl,T).
\end{equation}
Here each summand is free of rank one over $A$ and is maximal isotropic under the local Tate pairing.
 
 The Selmer structures above may be modified by introducing new primes.
 \begin{definition} Let $\mathcal{N}$ denote the set of squarefree products of primes in $\mathcal{L}$. For any $\fa\fb\mathfrak{c}\in\mathcal{N}$ we define a Selmer structure $({\mathcal{F}}^{\fa}_\fb(\mathfrak c), {\Sigma}_{{\mathcal{F}}^{\fa}_{\fb}(\mathfrak c)})$ as follows: ${\Sigma}_{{\mathcal{F}}^\fa_\fb(\mathfrak c)}$ is ${\Sigma}_{\mathcal{F}}$ together with all prime divisors of $\fa\fb\mathfrak{c}$,
\begin{equation*}
H^1_{{\mathcal{F}}^\fa_\fb(\mathfrak c)}(K_w,T)=
H^1_{\mathcal{F}}(K_w,T)
\end{equation*} 
for $w$ prime to $\fa\fb\mathfrak{c}$, and
\begin{equation*}
H^1_{{\mathcal{F}}^\fa_\fb(\mathfrak c)}(K_\fl,T)=
\begin{cases}
H^1(K_\fl,T) &\mbox{ if } \fl|\fa,\\ 
0 & \mbox{ if } \fl|\fb \\
H^1_{ord}(K_\fl,T) &\mbox{ if }\fl|\mathfrak c.
\end{cases}
\end{equation*}
Whenever any one of $\fa,\fb,\mathfrak c$ is the empty product, it is omitted from the notation.
\end{definition}
\begin{definition}\label{cartesian} A Selmer structure $(\mathcal{F}, {\Sigma}_{\mathcal{F}})$ is \emph{cartesian} if for every quotient $T/{\mm^i}T$ of $T$, every place $w{\in}{\Sigma}_{\mathcal{F}}$, and any generator ${\pi}{\in}\mm$, the isomorphism
\begin{equation*}
    T/{\mm^i}T {\lra} T[\mm^i]
\end{equation*}
induces an isomorphism $H^1_{\mathcal{F}}(K_w,T/\mm^iT){\cong}H^1_{\mathcal{F}}(K_w,T[\mm^i])$.
\end{definition}
For the remainder of this section, we make the following assumptions:
\begin{hypothesis}\label{key-hypothesis} 
\begin{enumerate}
    \item $(T/\mm T)^{G_K}=0$,
    \item $\mathcal{F}$ is cartesian.
\end{enumerate}
\end{hypothesis}
\begin{lemma}\label{sel[m]}
The Selmer structure $\mathcal{F}(\nn)$ is cartesian for any $\nn\in{\mathcal{N}}$. For any choice of generator ${\pi}{\in} \mm$ and any $0{\leq}i{\leq}\len(A)$, the composition 
\begin{equation*}
T/{\mm^i}T\xrightarrow{\pi^{\len{A}-i}} T[\mm^i]{\lra}T
\end{equation*}
induces isomorphisms $\Sel_{\mathcal{F}(\nn)}(K,T/\mm^i){\cong}\Sel_{\mathcal{F}(\nn)}(K,T[\mm^i]){\cong}\Sel_{\mathcal{F}(\nn)}(K,T)[\mm^i]$.
\end{lemma}
\begin{proof}
The cartesian structure of 
$\cF(\nn)$ is proved, exactly as in \cite[Lemma 3.5.4]{mr}, assuming that $(T/\mm T)^{G_K}=0$. We can see that $H^1(K,T/\mm^i){\cong}H^1(K,T[\mm^i]){\cong}H^1(K,T)[\mm^i]$,
from which the Lemma follows. 
\end{proof}
We now collect some of the results of Mazur and Rubin's (\cite{mr}) and Howard's on the structure of the Selmer groups. Note that the proof of the results do not require the irreducibility or reducibility of $T/\mm T$, but a modified form of the Cassels-Tate pairing along with the self duality hypotheses on $T$ and $\mathcal{F}$ .

\begin{prop}\label{selmer-facts}
\begin{enumerate}
    \item {\cite[2.2.7]{hbes}}\label{structure}
For any $\nn\in{\mathcal{N}}$ there is a (non-canonical) decomposition 
\begin{equation*}
    \Sel_{\mathcal{F}(\nn)}{\cong}A^{e(\nn)}{\oplus}M_\nn{\oplus}M_\nn\mbox{ with }e(\nn) \in\{0,1\}.
\end{equation*}
\item {\cite[Prop 2.2.9]{hbes}}\label{length} For any $\nn \fl \in \mathcal{N}$ there are non-negative integers $a,b$ with $a+b= \len(A)$ in the diagram of inclusions
\begin{equation*}
\xymatrix{
  & \Sel_{\mathcal{F}^{\fl}(\nn)}  \\
   \Sel_{\mathcal{F}(\nn)} \ar[ur]^b   &&   \Sel_{\mathcal{F}(\nn \fl)} \ar[ul]_a \\
  & \Sel_{\mathcal{F}_{\fl}({\nn})} \ar[ul]^a \ar[ur]_b
}
\end{equation*}
where the labels on the arrows are the lengths of the respective quotients. Here all the four quotients are cyclic $A$-modules and 
\begin{equation*}
    a=\len(\loc_{\fl}(\Sel_{\mathcal{F}(\nn)})), b=\len(\loc_{\fl}(\Sel_{\mathcal{F}(\nn \fl)})).
\end{equation*}
\item \cite[Cor 2.2.13]{hbes}\label{dim}
Fix $\nn{\in}{\mathcal{N}}$. Let $e(\nn)$ be as in \ref{structure},
and set
\begin{equation*}
\rho(\nn)=\dim_{A/\mm}(\Sel_{\mathcal{F}(\nn)}(K,T/\mm T)).
\end{equation*}
Then $e(\nn)\equiv{\rho(\nn)}\pmod{2}$, and for any $\fl{\in}{\mathcal{L}}$ prime to $\nn$
\begin{equation*}
\begin{split}
    \rho(\nn\fl)&={\rho}(\nn)+1\iff \loc_\fl(\Sel_{\mathcal{F}(\nn)}(K,T/\mm T))=0\\
\rho(\nn\fl)&={\rho}(\nn)-1\iff \loc_\fl(\Sel_{\mathcal{F}(\nn)}(K,T/\mm T)){\neq}0
\end{split}
\end{equation*}
Here $\rho(\nn)$ remains unchanged, and the equivalences hold if one replaces $\Sel_{\mathcal{F}(\nn)}(K,T/\mm T)$ by $\Sel_{\mathcal{F}(\nn)}(K,T)[\mm]$ everywhere.
\end{enumerate}
\end{prop} 
\begin{definition}
Let $\cN^{odd}$ denote the subset of $\cN$ for which $e(\nn)=1$ and $\cN^{even} \subset \cN$ is the subset for which $e(\nn)=0$ where $e(\nn)$ is as defined in the first statement of the Proposition \ref{selmer-facts}\ref{structure}.
\end{definition}
By \ref{dim} above, 
    for $\nn\fl{\in}{\mathcal{N}}$,
\begin{equation}
\nn{\in}{\mathcal{N}^{even}}\iff \nn\fl{\in}{\mathcal{N}^{odd}}.
\end{equation}
\begin{prop}\label{length+-}\cite[Cor 2.2.12]{hbes}
Let $\nn\fl{\in}{\mathcal{N}}$, and $a$ and $b$ be as in Proposition \ref{selmer-facts}\ref{length}. Then
\begin{equation*}
\len(M_\nn)=
\begin{cases}
\len(M_{\nn\fl})+a &\mbox{ if }\nn\in{\mathcal{N}^{even}}\\ \len(M_{\nn\fl})-b &\mbox{ if }\nn\in{\mathcal{N}^{odd}}.
\end{cases}
\end{equation*}
\end{prop}

\section{The Key Lemma}\label{section-key}
Let $\Gamma^-=\gal(K^{ac}/K)$. Then $\Gamma^-\cong\Z_p$, and we write $\gamma$ for a topological generator of $\Gamma^-$. 
Let $T_p$ be the Tate module of $E$, and $\rho_E:G_\Q{\lra}\operatorname{Aut}_{\Zp}(T_p)$ denote the representation attached to the elliptic curve $E$. We set $U={\Zp}^{\times}{\cap}\img(\rho_E)$. 
Let $R$ be the ring of integers of a finite extension of $\Q_p$, with maximal ideal $\mm$ and  $v_p$ is the normalized valuation on $R$. Consider a character $\alpha: \Gamma^-{\lra}{R^{\times}}$.  
Following \cite{nekovar, cgls}, we define 
\begin{eqnarray*}
 C_1&:=&\operatorname{min}\{v_p(u-1)\mid u{\in}U\}\\
 C_2&:=&\operatorname{min}\{m{\geq}0\mid p^m\End_{\Zp}(T_p
 ){\subset}\rho_E(\Zp[G_\Q])\}\\
 C_{\alpha}&:=& \begin{cases}
                           v_p(\alpha(\gamma)-1), &\mbox{ if } \alpha{\neq}1 \\ 
                           0, &\mbox{ otherwise}.
              \end{cases}
\end{eqnarray*}

\begin{remark}[{\cite[Remark 3.3.5]{cgls}}]
    As remarked by Castella \emph{et al.}, if $\rho_E$ is surjective, then clearly $C_1=0$, and if $E[p]$ is irreducible, then $C_2=0$. In particular, $C_1=C_2=0$ if $\rho_E$ is surjective.
\end{remark}

We write $\pi$ for a uniformizer of $\mm$. For $k \geq1$, let $R^{(k)}=R/\mm^kR$, $T^{(k)}=T_p/\mm^kT_p$ and suppose $\ell$ be a rational prime. For each $\fl |\ell{\in}{\mathcal{L}}$, let $I_\ell$ be the smallest ideal containing $(\ell+1)$ for which the Frobenius element $\frob_\fl{\in}G_{K_\fl}$ acts trivially on $T_p/{I_\fl}T_p$, $\mathcal{L}_k$=$\{\fl{\in}{\mathcal{L}} \mid I_\ell{\subset}p^k\Zp\}$ and let $\mathcal{N}_k$ be the set of square-free products of primes in $\mathcal{L}_k$. 

For brevity, we write 
\begin{enumerate}
\item $r:=\operatorname{rank}_{\Z_p}(R)$.
\item $\varepsilon_0:=r(C_1+C_2+C_\alpha)$.
\item $A:=R^{(k)}$, which is a principal local ring with maximal ideal denoted by $\mm$ again,
\item $T_\alpha:=T^{(k)}\otimes A(\alpha)$, which is the $G_K$-module $T^{(k)}$ twisted by a character $\alpha$.
\item $\cL:=\cL_k$.
\end{enumerate}

\begin{definition}\label{ord-exp-ind}
Let $M$ be a finitely generated $A$-module. Then
\begin{enumerate}
\item the \emph{order of} $x\in M$ is denoted by
\begin{equation*}
    \ord(x):= \min\{m \geq 0: {\pi}^m.x=0\},
\end{equation*}
\item the \emph{exponent of} $x\in M$ is denoted by 
\begin{equation*}
    \exp(M):=\min\{n\geq 0:{\pi}^nM=0\}=\max\{\ord(x):x\in M\},
\end{equation*}
\item 
the \emph{index of divisibility} of $x$ in $M$ is denoted by
\begin{equation*}
    \ind(x,M):=\operatorname{max}\{j \leq \infty | x \in {\mm}^jM\}.
\end{equation*}
\end{enumerate}
\end{definition}

Suppose $s:C\lra D$ be a surjective map of two
finitely generated $A$-modules, then for $x\in C$ 
\begin{equation}\label{surj-ind}
    \ind(x,C)\leq \ind(s(x),D).
\end{equation}
We continue to assume that Hypotheses \ref{key-hypothesis}  holds. 
\begin{theorem}[{\cite[Lemma 1.6.2]{how2}},{\cite[Theorem 3.2]{bd}}]\label{HBD}
Suppose $T^{(k)}/\mm T^{(k)}$ is irreducible as a representation of $G_K$, and $c \in H^1(K,T^{(k)}/{\mm}T^{(k)})$ is non-zero. Then there are infinitely many primes $\fl \in \cL$ such that $\loc_\fl(c) \neq 0$.
\end{theorem}
\begin{remark}
    It is to be noted that using the above Theorem, we do not need to twist $T^{(k)}$ by a non-trivial character $\alpha$ in the irreducible case. In other words, we can take $\alpha=1$ in this case, and $C_\alpha =0$.
\end{remark}
Building on results of Nekov\'a\v{r} \cite[Lemma 6.6.1(iii), Prop 6.1.2, Cor 6.3.4]{nekovar} the following result is proved by Castella \emph{et al}. 
\begin{theorem}[{\cite[Prop. 3.3.6]{cgls}}]\label{order} Suppose $\alpha\neq 1$ and $c_1,c_2 \in H^1(K,T_{\alpha})$ are cocyles such that 
$Ac_1+Ac_2$ contains a submodule isomorphic to ${\mm}^{d_1}A \oplus {\mm}^{d_2}A$ for some $d_1,d_2 \geq 0$.  Then for any $c_3\in H^1(K,T_\alpha)$ there exist infinitely many $\fl{\in}{\mathcal{L}_k}$ such that 
$\ord(\loc_\fl(c_3)){\geq}\ord(c_3)-\varepsilon_0$.
  
\end{theorem}

\begin{prop}\label{cocycle}

 Let $k>\varepsilon_0$ and $\alpha\neq1$. 
Then for any cyclic free $A$-submodule $C=Ac$ of rank one contained in $\Sel_{\cF(\nn)}(K,T_\alpha)$, there exists infinitely many primes $\fl\in\mathcal{L}$  such that $\loc_\fl(c)\neq0$.
\end{prop}
\begin{proof}
For brevity, we write $T$ for $T_\alpha$. 
Since
$\Sel_{\mathcal{F}(\nn)}(K,T){\cong}A^{e(\nn)}\oplus M_{\nn}{\oplus}M_{\nn}$, and 
it contains a non-zero submodule $C$, we have $e_\nn\neq0$ or $M_\nn\neq0$. It is enough to
prove the statement in the following cases.

 \emph{Case I:} Let $M_\nn\neq0$. Then  $\Sel_{\mathcal{F}(\nn)}(K,T){\supseteq}M_{\nn}{\oplus}M_{\nn}\neq0$. Let $x\in M_\nn$ be a non-zero cohomology class.  
 Let $x$ be not torsion. Then  
 $Ax$ is a free $A$-submodule, and $Ax\oplus Ax$ contains a 
 submodule isomorphic to $A\oplus A$.
 
 Let $x$ be $A$-torsion, and let $t=\ord(x)$. Then $t<k$, $\pi^tx=0$, and the natural map
 $A\lra Ax$ has kernel generated by $\pi^t$, and the
 natural surjective map $A\surj \pi^{k-t}A$ factors through $\pi^tA$. Therefore, $\pi^{k-t}A\cong A/\pi^{t}A\cong Ax$.
 
 Taking $c_1$ to be equal to $x$ coming from $M_\nn$ and $c_2$ from
 the other summand $M_\nn$, we see that the submodule $xA\oplus xA$ has a submodule isomorphic to $\mm^tA\oplus\mm^tA$.
Then we get two cohomology classes $c_1,c_2 \in \Sel_{\mathcal{F}}(K,T)$ such that 
$c_1A\oplus c_2A\supset \mm^tA\oplus\mm^tA$. 
Now by Theorem \ref{order} we have 
\begin{equation*}
    \ord(\loc_\fl(c)){\geq}\ord(c)-r(C_1+C_2+C_{\alpha}) \geq k-\varepsilon_0>0
\end{equation*}
Therefore, by Theorem \ref{order}, we have infinitely many $\fl$ such that $\loc_\fl(c) \neq 0$.

 \emph{Case II:} 
Let $e_\nn=1, M_\nn=0$. Then $\Sel_{\mathcal{F}(\nn)}(K,T)=Ax\cong A$, for some generator $x$.
We first show that  either $\Sel_{\mathcal F(\nn)}(K,T)\neq H^1(K,T)$ or $\loc_\fl(\Sel_{\mathcal{F}(\nn)}(K,T))\neq0$
for all $\fl\in\mathcal L\backslash \{\mbox{ primes dividing }\nn\}$. 

Towards a contradiction, suppose that, both the equalities hold simultaneously, i.e.,  
\begin{equation*}\begin{cases}
\Sel_{\mathcal F(\nn)}(K,T)= H^1(K,T), \mbox{ and }\\ \loc_\fl(\Sel_{\mathcal{F}(\nn)}(K,T))=0 \mbox{ for some } \fl\in\mathcal L\backslash \{\mbox{ primes dividing }\nn\}. 
\end{cases}
\end{equation*}
Here $\nn\fl\in\mathcal N$.    
By definition, 
\begin{equation*}
\Sel_{\mathcal{F}_\fl(\nn)}(K,T)=\Sel_{\mathcal{F}(\nn)}(K,T)\cap\Sel_{\mathcal{F}(\nn\fl)}(K,T)=H^1(K,T)\cap\Sel_{\mathcal{F}(\nn\fl)}(K,T)=\Sel_{\mathcal{F}(\nn\fl)}(K,T).    
\end{equation*}
 Then \begin{equation*}
\len(\loc_\fl(\Sel_{\mathcal{F}(\nn)}(K,T)))=0 \mbox{ and } \len(\loc_\fl(\Sel_{\mathcal{F}(\nn\fl)}(K,T)))=0.
\end{equation*}
This is a contradiction as these lengths add up to $\len(A)=k$ by \cite[Prop 2.2.9]{hbes}.
This proves our claim as $\Sel_{\mathcal{F}(\nn)}(K,T)\neq0$.

In case, $\Sel_{\mathcal F(\nn)}(K,T)\neq H^1(K,T)$, let $c_2\in H^1(K,T)$ such that it is not
zero in the quotient $H^1(K,T)/\Sel_{\mathcal F(\nn)}(K,T)$. Then the submodule generated
by $x$ and $c_2$ contains a submodule isomorphic to $A\oplus \mm^b$ for some $a,b\geq0$. Here again, for the free submodule $Ac=C$, by Theorem \ref{order}  we get infinitely many $\fl$ such that $\loc_\fl(c)\neq 0$.

In the other case, $\Sel_{\cF(\nn)}(K,T)=H^1(K,T)\cong A$, so $Ac=H^1(K,T)$, it is clear that we have infinitely many $\fl$ such that $\loc_\fl(c)\neq 0$.  
This completes the proof.
\end{proof}

We now have the following key lemma.
  \begin{lemma}\label{key-lemma} Let $\len(A)>\varepsilon_0$. Then for any $\nn\in \mathcal{N}$ and any cyclic free rank one $A$-submodule $C\subset\Sel_{\mathcal{F}(\nn)}(K,T)$, there exists infinitely many $\fl \in \mathcal{L}$ such that $\loc_\fl(C)\cong H^1_{unr}(K_\fl,T)$.
  \end{lemma}
 \begin{proof} Let $C$ be generated by $c$, i.e., $C=Ac \cong A$.
 Suppose $k=\len(A)$, then $\ord(c)=k$.  
 By the previous Proposition, there exists infinitely many $\fl\in\cL$ such that
 \begin{equation*}
     \ord(\loc_{\fl}(c)) \geq \ord(c)-\varepsilon_0 \geq k-\varepsilon_0 > 0.
 \end{equation*}
 Therefore, $\loc_{\fl}$ takes $C$ injectively into $H^1_{unr}(K,T)$, which by equation \eqref{decompo} is isomorphic to $A$. Hence $\loc_\fl(C) \cong H^1_{unr}(K,T)$. 
 \end{proof}
 \begin{corollary}\label{key-corollary}
 Let $\len(A)>\varepsilon_0$. Then for any $\nn\in \mathcal{N}$ and any free $A$-submodule $C\subset\Sel_{\mathcal{F}(\nn)}(K,T)$ of rank one, there exists infinitely many $\fl \in \mathcal{L}$ such that $C/\mm\cong H^1_{unr}(K,T/\mm T)$. In particular,
 $\loc_\fl(\Sel_{\cF(\nn)}(K,T/\mm T))\neq0$, for infinitely many primes $\fl\in\cL$.
 \end{corollary}
 \begin{proof}
 By the lemma above, for infinitely many primes $\fl\in\cL$, we have an isomorphism:
 \begin{equation*}
     C\stackrel{\loc_\fl}{\lra}H^1_{unr}(K,T)\cong A,
 \end{equation*}
 which induces the following isomorphism going modulo $\mm$:
 \begin{equation*}
     C/\mm\stackrel{\loc_\fl}{\lra}H^1_{unr}(K,T/\mm T)\cong A/\mm.
 \end{equation*}
For each of these primes $\fl$, as $C/\mm\inj\Sel_{\cF(\nn)}(K,T)/{\mm}\Sel_{\cF(\nn)}(K,T)\cong \Sel_{\cF(\nn)}(K,T/\mm T)$, we have 
 $\loc_\fl(\Sel_{\cF(\nn)}(K,T/\mm T))\neq0$.
 \end{proof}
 
 \begin{lemma}\label{mod-m}
 Let $\loc_\fl(\Sel_{\cF(\nn\fl)}(K,T))\neq 0$ for some $\nn\in {\cN}$ and $\fl\in\cL$, $\fl \nmid \nn$. Then, we have an isomorphism $\loc_\fl(\Sel_{\cF(\nn\fl)}(K,T))\cong A$.
 \end{lemma}
 
 \begin{proof}
 Consider the commutative diagram
 \begin{equation*}\label{mod-m-diagram}
     \xymatrixcolsep{5pc}\xymatrix{
    \Sel_{\cF(\nn\fl)}(K,T) \ar[r]^{\loc_\fl} \ar[d] & H^1_{ord}(K_\fl,T)\cong A \ar[d]\\
     \Sel_{\cF(\nn\fl)}(K,T/\mm) \ar[r]^{\loc_\fl} & H^1_{ord}(K_\fl,T/\mm)\cong A/\mm.
     }
 \end{equation*}
 Suppose  $\loc_\fl(\Sel_{\cF(\nn\fl)}(K,T))\ncong A$. Since $\loc_\fl(\Sel_{\cF(\nn\fl)}(K,T)) \neq 0$, we have $\loc_{\fl}(\Sel_{\cF(\nn\fl)}(K,T/{\mm}T))\neq 0$. So, we have $\loc_{\fl}(\Sel_{\cF(\nn\fl)}(K,T/{\mm}T))\cong A/\mm$, as the lower horizontal map is $A/\mm$-vector space homomorphism. By Proposition \ref{selmer-facts} \ref{length}, we have $\loc_\fl(\Sel_{\cF(\nn)}(K,T/\mm T))=0$. However, it can be seen that $\len(\loc_\fl(\Sel_{\cF(\nn\fl)}(K,T))\neq \len(A)$, hence  $\loc_\fl(\Sel_{\cF(\nn)}(K,T))\neq 0$. It follows that that  $\loc_\fl(\Sel_{\cF(\nn)}(K,T/\mm T))\cong A/\mm$, which is a contradiction. Thus  $\loc_\fl(\Sel_{\cF(\nn\fl)}(K,T))\cong A$.
 \end{proof}
 
\begin{remark}
   In the irreducible case, as mentioned in Theorem \ref{HBD}, the error term $\varepsilon_0$ does not appear. 
 \end{remark}

 \section{Bipartite Euler System over Artinian Rings}\label{bipartite}
We continue with the notations introduced in Sections \ref{selmer} and \ref{section-key}, along with the Hypothesis \ref{key-hypothesis}. We still continue to denote $T_\alpha$ by $T$. Following Howard \cite[Def 2.3.2]{hbes}, we define a bipartite Euler system as follows:
\begin{definition}\label{BES}  A \emph{bipartite Euler system of odd type} for $(T,\mathcal{F},\mathcal{L})$ is a pair of families
\begin{equation*}
   \{\kappa_\nn{\in}\Sel_{\mathcal{F}(\nn)}(K,T)\mid \nn\in{\mathcal{N}^{odd}}\}  \mbox{ and } \{\lambda_\nn{\in}A\mid\nn\in{\mathcal{N}^{even}}\}
\end{equation*}
related by the following reciprocity laws:
\begin{enumerate}
    \item for any $\nn\fl{\in}{\mathcal{N}^{odd}}$, there exists an isomorphism of $A$-modules
\begin{equation*}
    A/(\lambda_\nn)\cong H^1_{ord}(K_\fl,T)/A.\loc_\fl(\kappa_{\nn\fl}),
\end{equation*}
   \item for any $\nn\fl{\in} \mathcal{N}^{even}$, there exists an isomorphism of $A$-modules
\begin{equation*}
    A/(\lambda_{\nn\fl}) \cong H^1_{unr}(K_\fl,T)/A.\loc_\fl(\kappa_\nn).
\end{equation*}
\end{enumerate}
\end{definition}
A \emph{bipartite Euler system of even type} is defined in the same way, but with the even and odd term interchanged everywhere in the definition. An Euler system is said to be \emph{non-trivial} if $\lambda_\nn\neq0$ for some $\nn$. By the reciprocity law this is equivalent to saying $\kappa_{\nn\fl} {\neq} 0$ for some $\fl\in\cL$. 

 \begin{prop}\label{even type}
Let $\len(A)>\varepsilon_0$. Then there are no nontrivial Euler systems for $(T,\mathcal{F},\mathcal{L})$ of even type.
 \end{prop}
 \begin{proof} Suppose there exists a non-trivial Euler system of \emph{even} type. Then for some $\nn \in \mathcal{N}^{odd}$, $\lambda_\nn \neq 0$.
  As $e_{\nn}=1$, $\Sel_{\mathcal{F}(\nn)}$ contains a free $A$-module of rank 1, and 
  by Lemma \ref{key-lemma}, $\loc_\fl(\Sel_{\cF(\nn)}(K,T))\neq0$ for some $\fl\in\cL$.
  It follows that the injective map 
 \begin{equation*}
      \Sel_{\mathcal{F}(\nn)}(K,T)/\Sel_{\mathcal{F}_\fl(\nn)}(K,T) \lra H^1_{unr}(K_\fl,T) \cong A
\end{equation*}
  is an isomorphism. By Proposition \ref{length}, we get $\Sel_{\mathcal{F}(\nn)}(K,T)=\Sel_{\mathcal{F}(\nn\fl)}(K,T)$. Therefore, $\loc_\fl(\kappa_{\nn\fl})=0$, which is not possible by the reciprocity law.
\end{proof}  

\begin{prop}\label{annihilate}
Let $\len(A)=k>\varepsilon_0$, and consider a non-trivial Euler system of odd type for $(T,\mathcal{F},\mathcal{L})$. Let $\Sel_{\cF(\nn)}(K,T)\cong A^{e_\nn}\oplus M_{\nn}\oplus M_\nn$ be as in the Proposition 
\ref{selmer-facts}\ref{structure}. We have
\begin{enumerate}
    \item for $\nn \in \mathcal{N}^{even}$, if $\lambda_\nn\neq 0$, then $\mm^{k-1}M_{\nn}=0$, 
    \item for $\nn \in \mathcal{N}^{odd}$,  
    if $\kappa_\nn\neq0$ and is contained in a cyclic $A$-free rank one submodule $C$ of $\Sel_{\cF(\nn)}(K,T)$, 
    then $\mm^{k-1}M_{\nn}=0$. 
\end{enumerate}
\end{prop}  
\begin{proof} 
 (i) Let $\nn\in\cN^{even}$ such that ${\mm}^{k-1}M_{\nn} \neq 0$. This implies $\Sel_{\mathcal{F}(\nn)}(K,T)$ contains a free submodule of rank one, say $C$. Let $\fl \in \mathcal{L}$ such that $\fl \nmid \nn$ and 
 $\loc_{\fl}(C)\cong H^1_{unr}(K_{\fl},T)$ (using Lemma \ref{key-lemma}). By Proposition \ref{selmer-facts} \ref{length}, $\loc_{\fl}(\Sel_{\mathcal{F}(\nn \fl)}(K,T))=0$, from which we get $\loc_{\fl}(\kappa_{\nn \fl})=0$. By the first reciprocity laws, $\lambda_{\nn}=0$, which gives a contradiction.
 
 (ii) Let $\nn\in\cN^{odd}$ such that ${\mm}^{k-1}M_{\nn} \neq 0$. By Proposition \ref{selmer-facts}\ref{structure} we know $\Sel_{\mathcal{F}(\nn)}(K,T)$ contains a free submodule of rank two, say $D$. Thus by equation \eqref{decompo}, the kernel of the following map, 
\begin{equation*}
    \loc_{\fl}: \Sel_{\mathcal{F}(\nn)}(K,T){\lra}H^1_{unr}(K_{\fl},T)
\end{equation*}
which is $\Sel_{\mathcal{F}_{\fl}(\nn)}(K,T)$ contains a free submodule of rank one for any $\fl\in\cL$. As $\Sel_{\mathcal{F}_{\fl}(\nn)}(K,T) \subset \Sel_{\mathcal{F}(\nn \fl)}(K,T)$, and $\nn\fl\in\cN^{even}$, so $\mm^{k-1}M_{\nn\fl}\neq0$, otherwise $\mm^{k-1}\Sel_{\cF_\fl(\nn)}(K,T)=0$.  
By part (i), $\lambda_{\nn\fl}=0$. 
By the second reciprocity law, we have, $\loc_{\fl}(\kappa_{\nn})=0$ for all $\fl\in\cL, \fl\nmid\nn$. For $\kappa_\nn\neq0$, this gives a contradiction as there are infinitely many primes $\fl\in\cL$, such that $\loc_\fl(C)\cong A$ ( by Lemma \ref{key-lemma}). 
\end{proof}
\begin{definition} A non-trivial Euler system of \emph{odd} type is said to be \emph{free} if for every $\nn \in \mathcal{N}^{odd}$, $\kappa_\nn$ is contained in a cyclic submodule of $\Sel_{\cF(\nn)}(K,T)$ which is $A$-free of rank one.
\begin{definition}{\cite[Def 2.2.8]{hbes}}  
For $\nn{\in}{\mathcal{N}}$, we recall the definition of the \emph{stub module}
\begin{eqnarray*}
\stub_{\nn}&=& \begin{cases}
                {\mm}^{\len(M_{\nn})}A,  &  \mbox{ if } \nn \in {\cN}^{even} \\ 
                {\mm}^{\len(M_{\nn})}\Sel_{\cF(\nn)} , &\mbox{ if } \nn \in {\cN}^{odd}.
              \end{cases}
\end{eqnarray*}
\end{definition}
The stub modules for $\nn$ and $\nn\fl$ are related to each other in the following manner.

\begin{prop}[{\cite[Cor 2.2.13]{hbes}}]\label{loc stub}
Let $\nn\fl{\in}{\mathcal{N}}$. Then there is an isomorphism of $A$-modules 
\begin{equation*}
    \loc_\fl(\stub_{\nn}){\cong}\stub_{\nn\fl}, \mbox{ if } \nn\in{\mathcal{N}^{odd}}
\end{equation*} and
\begin{equation*}
\loc_\fl(\stub_{\nn\fl}){\cong}\stub_\nn, \mbox{ if } \nn\in{\mathcal{N}^{even}}.
\end{equation*}
\end{prop}
\end{definition}
Let $\rho_0:=\operatorname{max}\{\rho(\nn): \nn \in \cN\}$.
For $S$ a finite set of primes, let $G_S$ denotes the Galois group of the maximal extension of $K$ unramified outside $S$ over $K$. Then, it follows from the \emph{Hermite-Minkowski Theorem} that $\dim_{A/\mm}H^1(G_S,T/{\mm}T)$ is finite. So $\rho_0$ is bounded above by $\dim_{A/\mm}H^1(G_S,T/{\mm}T)$. We define $\varepsilon:=\varepsilon_0{\rho_0}$. Then $\varepsilon$ is a non-negative integer independent of the length of $A$.
\begin{theorem}\label{bound} 
Let $\len(A)>\varepsilon_0$.
For any free Euler system of odd type for $(T,\mathcal{F},\mathcal{L})$, ${\pi}^\varepsilon\lambda_\nn \in \stub_\nn$ for every $\nn \in \mathcal{N}^{even}$, and ${\pi}^\varepsilon\kappa_\nn \in \stub_\nn$ for every  $\nn \in \mathcal{N}^{odd}$. Equivalently, in terms of the $A$-module $M_\nn$ in the decomposition 
$\Sel_{\cF(\nn)}\cong A^{e_\nn}\oplus M_\nn\oplus M_\nn$, we have
\begin{equation*}
      \len(M_\nn) \leq \begin{cases}
          \ind(\lambda_\nn,A)+\varepsilon &\mbox{ if }\nn \in \mathcal{N}^{even}\\
          \ind(\kappa_\nn,\Sel_{\mathcal{F}(\nn)}(K,T))+\varepsilon &\mbox{ if } \nn \in \mathcal{N}^{odd}.
      \end{cases}
\end{equation*}
\end{theorem}
\begin{proof}  We prove this by induction on $\rho(\nn)$. If $\nn$ is even and $\rho(\nn)=0$, then $\len(M_\nn)=0$, so $M_\nn=0$. Hence $\stub_\nn=A$. Similarly if $\rho(\nn)=1$ and $\nn$ is odd, then again we have $M_\nn=0$. So the result follows. We now assume that $\rho(\nn)\geq 2$, and so $\len(M_\nn)\neq 0$.

First suppose that $\operatorname{exp}(M_\nn)\leq \varepsilon_0$. Then by the structure of $M_\nn$ as an $A$-module, we have, $\len(M_\nn)\leq {\rho_0}.\varepsilon_0$.

Now let us assume that $\operatorname{exp}(M_\nn)>\varepsilon_0$, and $\nn \in {\cN}^{even}$. Then $\Sel_{\cF(\nn)}(K,T)$ contains an element, say $c$, such that $\ord(\loc_\fl(c))>\ord(c)-\varepsilon_0>0$ for some $\fl\in \cL$, fix any $\fl \in \cL$ prime to $\nn$ such that $\loc_{\fl}(\Sel_{\cF(\nn)}(K,T/\mm T))\neq 0$. Also we suppose that $\lambda_\nn \neq 0$. By Proposition \ref{selmer-facts}\ref{dim} and the induction hypothesis, $\pi^\varepsilon\kappa_{\nn\fl} \in \stub_{\nn\fl}$, and so we have 
\begin{eqnarray*}
    \len(M_\nn)&=&\len(M_{\nn\fl})+a\\
              &\leq& \ind(\kappa_{\nn\fl},\Sel_{\cF(\nn\fl)})+a+{\rho_0}\varepsilon_0\\
              &\leq& \ind(\loc_\fl(\kappa_{\nn\fl}),\loc_\fl(\Sel_{\cF(\nn\fl)}))+a+{\rho_0}\varepsilon_0
\end{eqnarray*}
where $a$ is as in the Proposition \ref{selmer-facts}\ref{length}. 
By the first reciprocity law we have
\begin{eqnarray*}
    \ind(\lambda_\nn,A)&=& \ind(\loc_\fl(\kappa_{\nn\fl}),H^1_{ord}(K_\fl,T))\\
                      &=&  \ind(\loc_\fl(\kappa_{\nn\fl}),\loc_\fl(\Sel_{\cF(\nn\fl)}))+\len(H^1_{ord}(K_\fl,T)/\loc_{\fl}(\Sel_{\cF(\nn\fl)}))\\
                      &=& \ind(\loc_\fl(\kappa_{\nn\fl}),\loc_\fl(\Sel_{\cF(\nn\fl)})+\len(A)-b.
\end{eqnarray*}
Since by Proposition \ref{selmer-facts}\ref{length}, $a+b=\len(A)$, so $\len(M_\nn)\leq \ind(\lambda_\nn,A)+ \varepsilon$.

Now let us assume that $\exp(M_\nn)>\varepsilon_0$ and $\nn$ is odd. Let $C_\nn$ be a free $A$-module in $\Sel_{\cF(\nn)}$. Consider a prime $\fl \in \cL$ prime to $\nn$ such that $\loc_\fl$ takes $C_\nn$ isomorphically into $H^1_{unr}(K_\fl,T)$ ( using Lemma \ref{key-lemma}). By Proposition \ref{selmer-facts}\ref{dim} we have $\rho(\nn\fl)=\rho(\nn)-1<\rho(\nn)$. So, by induction hypothesis we have 
$\len(M_{\nn\fl}) \leq \ind(\lambda_{\nn\fl},A)+\varepsilon$.
By Proposition \ref{length+-} and second reciprocity law we have
\begin{eqnarray*}
\len(M_\nn)=\len(M_{\nn\fl}) &\leq& \ind(\lambda_{\nn\fl},A)+\varepsilon\\
           &=& \ind(\loc_\fl(\kappa_\nn),H^1_{unr}(K_\fl,T))+\varepsilon\\
           &=& \ind(\kappa_\nn,\Sel_{\cF(\nn)})+\varepsilon.
\end{eqnarray*}
\end{proof}
\begin{remark}
\begin{enumerate}
\item A proof of the above length bound is obtained in \cite[Theorem 2.3.7]{hbes}, assuming that $T/\mm T$ is an irreducible representation of $G_K$. There the error term $\varepsilon$ does not occur. The above theorem subsumes the result of Howard by taking $\varepsilon=0$ if $T/\mm T$ is irreducible.
    \item The above theorem can be compared with \cite[Theorem 3.2.1]{cgls}. An error term is also present there.
\end{enumerate}
\end{remark}
\section{Stub Modules}\label{section-stub}
We briefly recall some facts about Stub modules from \cite{hbes}. As in the previous section, we denote the twist $T_\alpha$ by $T$, and retain the notations in the previous section. Our definition of a \emph{core} vertex ( Def \ref{def-core}) is different from that of Howard or Mazur-Rubin's in \cite{mr} and we call them absolute core vertices.
Let $\cX:=(\mathcal{V}, \mathcal{E})$ be a graph with the set of vertices $\mathcal{V}:=\{v(\nn)\mid\nn\in\cN\}$ and the set of edges $\mathcal{E}:=\{e(\nn,\nn\fl)\mid \fl\in\cL\}$.
A vertex $v(\nn)$ is called even ( resp. odd) if $\nn\in\cN^{even}$ (resp. $\nn\in\cN^{odd}$). We often write $\nn$ is an even or odd vertex accordingly as $v(\nn)$ is even or odd vertex. Attached to this graph is an \emph{Euler System Sheaf of $A$-modules}, which is defined for a vertex $v=v(\nn)$ and an edge $e=e(\nn,\nn\fl)$ as:
\begin{eqnarray*}
    \ES(v)= \begin{cases} \Sel_{\mathcal{F(\nn)}}  &\mbox{ if } \nn \in \mathcal{N}^{odd}\\
    A &\mbox{ if } \nn \in \mathcal{N}^{even}
    \end{cases}
    \quad\mbox{ and }\quad
    \ES(e)= \begin{cases} 
                        H^1_{unr}(K_{\fl},T) &\mbox{ if } \nn \in \mathcal{N}^{odd}\\
                        H^1_{ord}(K_{\fl},T) &\mbox{ if } \nn \in \mathcal{N}^{even}.
            \end{cases}
\end{eqnarray*}
If $e=e(\nn,\nn\fl)$ is an edge with end point $v$, we recall the \emph{vertex-to-edge} map
\begin{equation*}
    {\psi}^e_v: \ES(v){\lra} \ES(e)
\end{equation*}
defined as follows. If $v$ is odd then
\begin{equation*}
    {\psi}^e_v= \loc_{\fl} :\begin{cases} \Sel_{\mathcal{F}(\nn)}{\lra}H^1_{unr}(K_{\fl},T) &\mbox{ if } v=v(\nn)\\
    \Sel_{\mathcal{F}(\nn\fl)}{\lra}H^1_{ord}(K_{\fl},T) &\mbox{ if } v=v(\nn\fl).
    \end{cases}
\end{equation*}
If $v$ is even then we fix, using equation \eqref{decompo}, an isomorphism
\begin{equation*}
    {\psi}^e_v:A{\cong} \begin{cases} H^1_{unr}(K_{\fl},T) &\mbox{ if } v=v(\nn\fl)\\
    H^1_{ord}(K_{\fl},T) &\mbox{ if } v=v(\nn).
    \end{cases}
\end{equation*}
Here we fix a choice of this isomorphism for each edge $e$ and even vertex
$v$. 

Over each vertex $v=v(\nn)$ and edge $e=e(\nn,\nn\fl)$, the \emph{stub sheaf}
$\stub(\cX)$ is defined, as follows
\begin{equation*}
    \stub(v):=\stub_{\nn}{\subset}\ES(v) \quad \mbox{ and } \quad
    \stub(e):= \begin{cases} \loc_\fl(\stub_{\nn}) &\mbox{ if } \nn \in \mathcal{N}^{odd}\\
    \loc_\fl(\stub_{\nn\fl}) &\mbox{ if } \nn \in \mathcal{N}^{even}.
    \end{cases}
\end{equation*}
If $v$ is an even vertex, and $v'$ odd with edge $e=e(v,v')$, then the vertex-to-edge map 
${\psi}^e_{v'}:\stub(v'){\lra}\stub(e)$ is surjective. By Corollary \ref{loc stub}, it can be seen that the map ${\psi}^e_v$ gives an isomorphism 
$\stub(v){\cong}\stub(e)$.

\begin{definition} \label{def-core}
\begin{enumerate}
\item Let  $u=\min\{\len(M_\nn)|\nn \in \cN\}$. Then a vertex $v$ of $\cX$ is called an \emph{absolute core vertex} if $\stub(v){\cong}{\mm}^uA$.

\item The \emph{absolute core subgraph} $\cX_{\operatorname{abs}} \subset \cX$ is the graph whose vertices are the absolute core vertices of $\cX$, with two vertices in $\cX_{\operatorname{abs}}$ connected by an edge in $\cX_{\operatorname{abs}}$ if and only if they are connected by an edge in $\cX$. We let $\stub(\cX_{\operatorname{abs}})$ be the restriction of $\stub(\cX)$ to $\cX_{\operatorname{abs}}$.
\end{enumerate}
\end{definition}
\begin{remark}
    In the special case when $T/\mm T$ irreducible, then by \cite[Lemma 2.4.9]{hbes}, the minimal length $u=0$, and we recover the definition of a \emph{core} vertex in \cite[Def 2.4.2]{hbes} and \cite[Def 4.1.8]{mr}. In other words, in the irreducible case \emph{absolute core} vertices are those vertices $\nn \in \cN$ such that $\Stub_\nn \cong A$ (see \cite{hbes}).
\end{remark}
\begin{definition}\label{def-path}
\begin{enumerate}
\item
A \emph{path} from a vertex $v$ to a vertex $v'$ in $\cX$ is a finite sequence of vertices $v=v_0,v_1,...,v_r=v'$ such that $v_i$ is connected to $v_{i+1}$ by an edge $e_i$. A path is surjective (for the locally cyclic sheaf $\stub(\cX)$) if the vertex-to-edge map
\begin{equation*}
    {\psi}^{e_i}_{v_{i+1}}: \stub(v_{i+1}){\lra} \stub(e_i)
\end{equation*}
is an \emph{ isomorphism} for every $i$. We define a \emph{path} in $\cX_{\operatorname{abs}}$ in the same way.
\item\label{def-surjective}
The graph is said to be \emph{path connected} if for any two vertices in the graph, there exists a \emph{surjective} path between them.
By \cite[Lemma 2.4.8]{hbes}, a \emph{path} $v_0,.....,v_r$ in $\cX$ is \emph{surjective} if and only if for every $i$
\begin{equation*}
    \len(\stub(v_{i+1})) \leq \len(\stub(v_i)).
\end{equation*}
\end{enumerate}
\end{definition} 
\begin{definition}\label{universally-trivial}
\begin{enumerate}
   \item 
Let $\nn\in{\cN}^{even}$, then we say that the vertex $v(\nn)$ is \emph{universally trivial} if
\begin{align*}
&\loc_\fl(\Sel_{\cF(\nn)})=0 \,\forall\, \fl\nmid \nn,\\
&\loc_{\fl_1}(\Sel_{\cF(\nn\fl)})\cong A \mbox{ for some } \fl_1\nmid\nn\fl\implies
\loc_{\fl_2}(\Sel_{\cF(\nn\fl\fl_1)})=0\,\forall\, \fl_2\nmid \nn, \\
&\cdots\quad \cdots\quad \cdots\quad,\\
&\loc_{\fl_{2k-1}}(\Sel_{\cF(\nn\fl\fl_1...\fl_{2k-2})})\cong A \mbox{ for some } \fl_{2k-1}\nmid\nn\fl\cdots\fl_{2k-2}\implies
\loc_{\fl_{2k}}(\Sel_{\cF(\nn\fl\fl_1...\fl_{2k-1})})=0 \,\forall\,\fl_{2k}\nmid \nn,\\
&\cdots\quad \cdots\quad \cdots\quad\\
&\mbox{and the process continues indefinitely.}
\end{align*}

\item 
Let $\nn\in {\cN}^{odd}$, then we say that the vertex $v(\nn)$ is 
\emph{ universally trivial } if 
\begin{align*}
&\loc_{\fl_1}(\Sel_{\cF(\nn)})\cong A \mbox{ for some }  \fl_1 \nmid \nn \implies
\loc_{\fl_2}(\Sel_{\cF(\nn\fl_1)})=0 \,\forall\, \fl_2\nmid \nn,\\
&\loc_{\fl_3}(\Sel_{\cF(\nn\fl_1\fl_2)})\cong A \mbox{ for some } \fl_3 \nmid \nn \implies
\loc_{\fl_4}(\Sel_{\cF(\nn\fl_1\fl_2\fl_3)})=0 \,\forall\, \fl_4\nmid \nn,\\
&\cdots\quad \cdots\quad \cdots\,\\
&\loc_{\fl_{2r-1}}(\Sel_{\cF(\nn\fl_1...\fl_{2r-2})})\cong A \mbox{ for some } \fl_{2r-1} \nmid \nn, \implies
\loc_{\fl_{2r}}(\Sel_{\cF(\nn\fl_1...\fl_{2r-1})})=0 \,\forall\, \fl_{2r} \nmid \nn\\
&\cdots\quad \cdots\quad \cdots\quad
\end{align*}
and the process continues indefinitely.
\end{enumerate}
\end{definition}
\begin{remark}
By Lemma \ref{key-lemma}, for any \emph{odd} vertex $v(\fa)$, there exists infinitely many $\fl\in\cL$ such that $\loc_\fl(\Sel_{\cF(\fa)})\cong A$, so the above process continues indefinitely and both the notions are indeed well-defined.
\end{remark}
\begin{prop}\label{universal}
\begin{enumerate}
    \item If $\len(M_\nn)=u$, then $\nn$ is universally trivial.
    \item Let $\nn$ and $\fa$ be any two vertices in $\cN$ which are universally trivial. Then 
\begin{equation*}
\len(M_{\nn})= \len(M_\fa).
\end{equation*}
\end{enumerate}
\end{prop}
\begin{proof}
(i) First, let $\nn$ be even, such that $\len(M_\nn)=u$, and $\nn$ is not universally trivial. Then, there are primes $\fl, \fl_1,\cdots,\fl_{2k}$ for some $k\in\N$, such that 
\begin{eqnarray*}
&\loc_\fl(\Sel_{\cF(\nn)})=0 \quad \forall \quad \fl\nmid \nn\\
&\mbox{ if } \fl_1\nmid\nn \mbox{ with }\loc_{\fl_1}(\Sel_{\cF(\nn\fl)})\cong A 
\mbox{ then } \loc_{\fl_2}(\Sel_{\cF(\nn\fl\fl_1)})=0 \quad\forall\quad  \fl_2\nmid \nn \\
&\cdots\quad\cdots\quad\cdots\quad\cdots\\
&\fl_{2k-1}\nmid\nn\fl\fl_1...\fl_{2k-2} \mbox{ with } \loc_{\fl_{2k-1}}(\Sel_{\cF(\nn\fl\fl_1...\fl_{2k-2})})\cong A \mbox{ but}
\loc_{\fl_{2k}}(\Sel_{\cF(\nn\fl\fl_1...\fl_{2k-1})})\neq0 \mbox{ with  }\fl_{2k}\nmid \nn.
\end{eqnarray*}
Then it is easy to see from Proposition \ref{length+-} that $\len(M_\nn)=\len(M_{\nn\fl})=...=\len(M_{\nn\fl\fl_1...\fl_{2k-1}})$, and 
\begin{equation*}
    \len(M_{\nn\fl...\fl_{2k-1}})=\len(M_{\nn\fl...\fl_{2k}})+a_{2k}
\end{equation*}
where $a_{2k}$ is as in the Proposition \ref{length+-} and $a_{2k}\neq 0$ (by the above assumption).
This contradicts that $u$ is minimal. Hence $\nn$ is universally trivial.

When $\nn$ is odd the proof is similar to the above one. Indeed, let us consider $\nn$ is odd, $\len(M_\nn)=u$, but not universally trivial. Then there exists primes $\fl_1,...,\fl_{r}$ for some $r\in \N$, such that
\begin{eqnarray*}
\loc_{\fl_1}(\Sel_{\cF(\nn)})\cong A, \loc_{\fl_2}(\Sel_{\cF(\nn{\fl_1})})=0,..., \loc_{\fl_{2r-1}}(\Sel_{\cF(\nn\fl_1\fl_2...\fl_{2r-2})})\cong A 
\end{eqnarray*}
but
\begin{equation*}
\loc_{\fl_{2r}}(\Sel_{\cF(\nn\fl_1\fl_2...\fl_{2r-1})})\neq 0.
\end{equation*}
By Proposition \ref{length+-}, we have
\begin{equation*}
    \len(M_\nn)=\len(M_{\nn\fl_1})=...=\len(M_{\nn\fl_1...\fl_{2r-1}})
\end{equation*}
and 
\begin{equation*}
    \len(M_{\nn\fl_1...\fl_{2r-1}})= \len(M_{\nn\fl_1...\fl_{2r}})+a_{2r}
\end{equation*}
where $a_{2r}$ is a non zero positive integer. This contradicts that $\len(M_\nn)=u$ is minimal. This completes the proof of the first part of the Proposition.

Before we proceed to prove the second part of this Proposition, we prove the lemma below.
\begin{lemma}
Let $\fa=\fl_1\fl_2...\fl_s$ be a universally trivial vertex. Let $\nn$ be any vertex such that $\len(M_\nn)=u$. Then there exists a universally trivial vertex $\nn\fa\fa'$ such that $\len(M_\nn)=\len(M_{\nn\fa\fa'})$.
\end{lemma}
\begin{proof}
\emph{Case-I:} Let $\nn$ be \emph{even}. Then by Proposition \ref{length+-}, $\len(M_\nn)=\len(M_{\nn\fl_1})$ and $\nn\fl_1$ is odd and universally trivial. By Lemma \ref{key-lemma} there exists a prime $\fl_1'$ such that $\len(M_{\nn\fl_1})=\len(M_{\nn\fl_1\fl_1'})$ (Proposition \ref{length+-}) and $\nn\fl_1\fl_1'$ is even and universally trivial. Proceeding like this and adding primes $\fl_2, \fl_2',\cdots,\fl_s$, we get a universally trivial vertex $\nn\fa\fa'$ such that $\len(M_\nn)=\len(M_{\nn\fa\fa'})$ for some $\fa'$.

\emph{Case-II:} Let $\nn$ be \emph{odd}. Then by Lemma \ref{key-lemma}, there exists a prime $\fl_1'$ such that $\len(M_\nn)=\len(M_{\nn\fl_1'})$ ( by Proposition \ref{length+-}). Here as $\nn\fl_1'$ is even, universally trivial, by using Proposition \ref{length+-} again, we have $\len(M_{\nn\fl_1'})=\len(M_{\nn\fl_1'\fl_1})$. Continuing like this we get $\len(M_\nn)=\len(M_{\nn\fa\fa'})$ for some $\fa' \in \cN$ as $\nn\fl_1\fl_1'$ is again a universally trivial vertex.
\end{proof}
(ii) We now continue with the proof of Proposition \ref{universal}. Recall that $M_\nn$ is of minimal length $u$, hence $M_{\nn\fa\fa'}$ is minimal by the previous lemma.
We show the second part of the proposition by induction on the number of prime factors of $\nn\fa'$ by constructing a \emph{surjective} path from $\fa$ to $\fa\nn\fa'$ such that $\len(M_{\nn\fa\fa'})=\len(M_\fa)$. Let the number of prime factors of $\nn\fa'$ be greater than one.

\emph{Case-Even:} Suppose $\fa\nn\fa'$ is \emph{even}. Then by Proposition \ref{length+-}, $\len(M_{\fa\nn\fa'/\fl})=\len(M_{\fa\nn\fa'})-b$. Since $\len(M_{\nn\fa\fa'})$ is minimal, $b=0$, so $\fa\nn\fa'/\fl$ is universally trivial, by part (i). By induction hypothesis there is a surjective path from $\fa$ to $\fa\nn\fa'/\fl$ such that $\len(M_\fa)=\len(M_{\fa\nn\fa'/\fl})$. 
Therefore $\len(M_\fa)=\len(M_{\nn\fa\fa'})$.

\emph{Case-Odd:} Suppose $\fb=\fa\nn\fa'$ is \emph{odd}. If there exists $\fq|\nn\fa'$ such that $\loc_\fq(\Sel_{\cF(\fb)})\cong A$, then by Proposition \ref{length+-} we have, $\len(M_{\fb/\fq})=\len(M_{\fb})$ and by induction hypothesis the result follows.

Let $\loc_\fq(\Sel_{\cF(\fb)})\ncong A$ for all $\fq|\nn\fa'$.
Let $\fl\nmid \fb$ be such that $\loc_\fl(\Sel_{\cF(\fb)})\cong A$. By Proposition \ref{length+-}, we have $\len(M_\fb)=\len(M_{\fb\fl})=u$. Since $\fb\fl$ is even, by minimality, we have $\len(M_{\fb\fl/\fq})=u$ for any $\fq|\nn\fa'$. It can be seen that $v(\fb), v(\fb\fl), v(\fb\fl/\fq)$ is a surjective path such that $\len(M_\fb)=\len(M_{\fb\fl})=\len(M_{\fb\fl/\fq})$. Again if we have one such prime $\tau \mid {(\fb\fl/\fq)}$ such that $\loc_\tau(\Sel_{\cF(\fb\fl/\fq)}(K,T)\cong A$, then by induction hypothesis we are done. Now suppose $\loc_\tau(\Sel_{\cF(\fb\fl/\fq)})\ncong A$ for all $\tau|(\fb\fl/\fq)$. Then, writing $\fb'=\fb\fl/\fq$, we have $\loc_\tau(\Sel_{\cF(\fb\fl/\fq)})=0$ for all $\tau|\fb'$ ( by Lemma \ref{mod-m}). Then
\begin{equation*}
    \Sel_{\cF(\fb')}=\Sel_{\cF_{\nn\fa'\fl/\fq}(\fa)}\subset \Sel_{\cF_{\nn\fa'/\fq}(\fa)}=\Sel_{\cF(\fb)}.
\end{equation*}
Since the length of $\Sel_{\cF(\fb)}$ on the right is minimal, so equality holds everywhere, and this contradicts $\loc_\fl(\Sel_{\cF(\fb)})\cong A$. Therefore, there exists 
$\tau$ such that $\loc_\tau(\Sel_{\cF(\fb')})\cong A$. Hence $\len(M_{\fb'})=\len(M_{\fb'/\tau})$, and  
$v(\fb),v(\fb\fl), v(\fb\fl/\fq), v(\fb\fl/{\fq\tau})$ is a surjective path such that
\begin{equation*}
\len(M_\fb)=\len(M_{\fb\fl})=\len(M_{\fb\fl/\fq})=\len(M_{\fb\fl/{\fq\tau}}).
\end{equation*}
By induction hypothesis the result follows.
\end{proof}

\begin{remark}
Note that Proposition \ref{universal} says that $\nn$ is an \emph{absolute core} vertex if and only if it is universally trivial.
\end{remark}
\begin{lemma}\label{core}
Let $v$ be any vertex of $\cX$. Then there is a absolute core vertex $v_0$ and a surjective path in $\cX$ from $v_0$ to $v$. Moreover, for any $\nn \in \mathcal{N}$ there is a vertex ${\nn}' \in \mathcal{N}$ with $\nn | {\nn}'$ such that $v({\nn}')$ is a absolute core vertex. This ${\nn}'$ may be chosen either in $\mathcal{N}^{even}$ or in $\mathcal{N}^{odd}$.
\end{lemma}
\begin{proof}
We construct a sequence of vertices $w_i=w(\nn_i)$ inductively, where $w_0=v$.
If $w_i=w(\nn_i)$ is absolute core then we are done. So let us first assume that $w_i$ is even and not absolute core. Then by the definition of absolute core vertex, there exists primes $\fl_1,\fl_2,\cdots, \fl_{2k}$ such that 
\begin{flalign*}
\loc_\fl(\Sel_{\cF(\nn_i)})=0&\quad\forall\quad\fl\nmid \nn_i,\\
\loc_{\fl_1}(\Sel_{\cF(\nn_i\fl)})\cong A, &\mbox{ and }
\loc_{\fl_2}(\Sel_{\cF(\nn_i\fl\fl_1)})=0 
\mbox{ with }\quad\fl_2\nmid \nn_i \\
\cdots\quad\cdots\quad\cdots&\quad\cdots\quad\cdots\\
\loc_{\fl_{2k-1}}(\Sel_{\cF(\nn_i\fl\fl_1...\fl_{2k-2})})&
\cong A, 
\mbox{ but }
\loc_{\fl_{2k}}(\Sel_{\cF(\nn_i\fl\fl_1...\fl_{2k-1})})\neq0 \mbox{ with } \fl_{2k}\nmid \nn_i.
\end{flalign*}
By Proposition \ref{length+-} we have $\len(M_{\nn_i})=\len(M_{\nn_i\fl\fl_1})=...=\len(M_{\nn_i\fl\fl_1...\fl_{2k-1}})$.
Since $\loc_{\fl_{2k}}(\Sel_{\cF(\nn_i\fl\fl_1...\fl_{2k-1})})\neq0$ for $\fl_{2k}\nmid \nn_i$, we have 
\begin{equation*}
    \len(M_{\nn_i\fl\fl_1...\fl_{2k-1}})=\len(M_{\nn_i\fl\fl_1...\fl_{2k}})+a
\end{equation*}
where $a>0$. Setting $w_{i+1}=w({\nn_i}\fl\fl_1...\fl_{2k})$, we get 
\begin{equation*}
    \len(\Stub(w_i))<\len(\Stub(w_{i+1})).
\end{equation*}
If $w_i=w(\nn_i)$ is odd, using Lemma \ref{key-lemma} choose $\fl \in \cL$ prime to $\nn_i$ such that localization at $\fl$ takes a free rank one submodule of $\Sel_{\cF(\nn_i)}$ isomorphically onto $H^1_{unr}(K_{\fl},T)$, and set $w_{i+1}=w({\nn_i}\fl)$. From the Proposition \ref{selmer-facts}\ref{dim} we have 
\begin{equation*}
    \len(\Stub(w_i))=\len(\Stub(w_{i+1})),
\end{equation*}
since $a=\len(A)$ and $b=0$.
Continuing this way we get a desired path from a absolute core vertex $v_0=w_k$ to any given vertex $v$, and also we have 
\begin{equation*}
    \len(\Stub(w_i)) \leq \len(\Stub(w_{i+1}))
\end{equation*}
for all $i$. By \cite[Lemma 2.4.8]{hbes} the path $w_k,...,w_0$ is a surjective path from $v_0$ to $v$. Clearly, if $v_0$ is absolute core and even (resp. odd), then by the definition of universally trivial, it can be connected to an odd (resp. even) absolute core vertex.
\end{proof}
For any $\fa \in \mathcal{N}$, let $\cX_{\operatorname{abs},\fa}$ be the subgraph of $\cX_{\operatorname{abs}}$ whose vertices consist of those absolute core vertices $v(\nn)$ with $\fa|\nn$. Two vertices are connected by an edge in $\cX_{\operatorname{abs},\fa}$ if and only if they are connected by an edge in $\cX_{\operatorname{abs}}$.
\begin{corollary}\label{path conn}
Let $\len(A)>\varepsilon_0$ and $v(\fa)$ be a absolute core vertex. Then the subgraph $\cX_{\operatorname{abs},\fa}$ is path connected.
\end{corollary}
\begin{proof}
    Follows from the proof of the Proposition \ref{universal}.
\end{proof}

\begin{prop}\label{path}
The absolute core subgraph $\cX_{\operatorname{abs}}$ is path connected and contains both even and odd vertices. For any vertex $v$ of $\cX$ and any absolute core vertex $v_0$ of $\cX$, there is a surjective path from $v_0$ to $v$.
\end{prop}
\begin{proof}
The proof goes along similar lines as in \cite[Prop 2.4.11]{hbes}. That the absolute core subgraph contains both even odd vertices is clear from the proof of Lemma \ref{core}. 
Consider $v(\fa)$ and $v(\fb)$ two vertices of the absolute core subgraph $\cX_{\operatorname{abs}}$. By Lemma \ref{core}, we may choose $\nn \in \mathcal{N}$ divisible by $\operatorname{lcm}(\fa,\fb)$ such that $v(\nn)$ is absolute core vertex. By Corollary \ref{path conn} we know there is a path from $v(\fa)$ to $v(\nn)$ in $\cX_{\operatorname{abs},\fa}$ and similarly a path from $v(\fb)$ to $v(\nn)$ in $\cX_{\operatorname{abs},\fb}$. Since any path in $\cX_{\operatorname{abs},\fa}$ and $\cX_{\operatorname{abs},\fb}$ is also a path in $\cX_{\operatorname{abs}}$, there is a path from $v(\fa)$ to $v(\fb)$. Since any path in $\cX_{\operatorname{abs}}$ is surjective, any two absolute core vertices may be connected by a surjective path.  
Now given any $v$ in $\cX$ by Lemma \ref{core} we know there exist a absolute core vertex $v_0$ and a surjective path from $v$ to $v_0$.
\end{proof}
\begin{definition}
A \emph{global section}, $s$, of the sheaf $\stub(\cX)$ on $\cX$ is a function on vertices and edges of $\cX$,
\begin{equation*}
\begin{split}
    v \mapsto s(v) \in \stub(v)\\
    e \mapsto s(e) \in \stub(e),
\end{split}
\end{equation*}
such that for every edge $e$ with end points $v$ and $v'$
\begin{equation*}
    {\psi}^e_v(s(v))=s(e)={\psi}^e_{v'}(s(v'))
\end{equation*}
in $\stub(e)$. Similarly a global section of $\ES(\cX)$ can be defined.
\end{definition}
\begin{corollary}\label{gen}
For any global section $s$ of $\stub(\cX)$ there is a unique $\delta_0 =\delta(s)$ with $0\leq\delta_0\leq \len(A)-u$ such that $s(v)$ generates ${\mm}^{\delta_0}\stub(v)$ for every vertex $v$ of $\cX$. The global section $s$ is determined by its value at any absolute core vertex.
\end{corollary}
\begin{proof}
Let $v_0$ be a absolute core vertex. Define $\delta_0$ to be such that $s(v_0)$ generates ${\mm}^{\delta_0}\stub(v_0)$. Since $v_0$ is absolute core, we have $0\leq\delta_0\leq\len(A)-u$. By Proposition \ref{path}, for any vertex $v$ in $\cX$ consider a surjective map $\stub(v_0){\lra}\stub(v)$ taking $s(v_0)$ to $s(v)$. Then the image $s(v)$
generates $\mm^{\delta_0}\stub(v)$. 
\end{proof}
We now have the rigidity theorem below. This may be compared with \cite[Th 2.5.1]{hbes}.
\begin{theorem}\emph{(The Rigidity theorem)}\label{Rigidity}
Assume the Hypotheses \ref{key-hypothesis}. In addition, assume that we have a nontrivial free Euler system of odd type for $(T,\mathcal{F},\mathcal{L})$ such that $\pi^\varepsilon\lambda_\nn\neq0$. Then there is a unique integer $\delta$, independent of $\nn \in \mathcal{N}$, such that  ${\pi}^\varepsilon\lambda_{\nn}$ generates ${\mm}^{\delta}\stub_{\nn}$ for every $\nn \in \mathcal{N}^{even}$ and ${\pi}^\varepsilon\kappa_{\nn}$ generates ${\mm}^{\delta}\stub_{\nn}$ for every $\nn \in \mathcal{N}^{odd}$. Furthermore, $\delta$ is given by 
\begin{eqnarray*}
    \delta &=& \min\{\ind(\pi^\varepsilon\lambda_{\nn},{\mm}^uA)| \nn \in \mathcal{N}^{even}\}\\
           &=& \min\{\ind(\pi^\varepsilon\kappa_\nn,{\mm}^u\Sel_{\cF(\nn)})| \nn \in \cN^{odd}\}.
\end{eqnarray*}
\end{theorem}
\begin{proof}
Let $v(\nn), v(\nn\fl)$ be two vertices of the graph $\cX$, with edge $e=e(\nn,\nn\fl)$. Consider the global sections $s(v), s(e)$ of $\stub(\cX)$ defined by
\begin{eqnarray*}
    s(v)= \begin{cases} {\pi}^\varepsilon\lambda_{\nn} \mbox{ if } \nn \in \mathcal{N}^{even}\\
    {\pi}^\varepsilon\kappa_{\nn} \mbox{ if } \nn \in \mathcal{N}^{odd}
    \end{cases}
    &\quad \mbox{ and }\quad&
    s(e)= \begin{cases} {\pi}^\varepsilon\loc_\fl(\kappa_{\nn}) \mbox{ if } \nn \in \mathcal{N}^{odd}\\
    {\pi}^\varepsilon\loc_\fl(\kappa_{\nn \fl}) \mbox{ if } \nn \in \mathcal{N}^{even}.
    \end{cases}
\end{eqnarray*}
These maps are well defined by Theorem \ref{bound}, and $s$ is a non-trivial global section of the Euler system sheaf $ES(\cX)$.
By Theorem \ref{bound} this global section is actually a global section of the subsheaf $\stub(\cX)\subset \ES(\cX)$. By Corollary \ref{gen}, there is a unique $0\leq \delta< \len(A)-u$ such that $s(v)$ generates ${\mm}^{\delta}\stub(v)$ for every vertex $v$. For any other vertex $v$ with $s(v) \neq 0$ we have 
\begin{equation*}
    \delta= \ind(s(v),\stub(v)) \leq \ind(s(v), {\mm}^u\ES(v))=\begin{cases}
    \ind(s(v), \mm^uA), &\mbox{ if } v \mbox{ is even}\\
    \ind(s(v), \mm^u\Sel_{\cF(\nn)}), &\mbox{ if } v \mbox{ is odd}
    \end{cases}
\end{equation*}
with equality if and only if $\len(M_\nn)=u$, i.e.,  $v$ is absolute core. As there are even absolute core vertices by  Proposition \ref{core}, we have
\begin{equation*}
    \delta= \min\{ \ind(s(v),{\mm}^u\ES(v)) | v \mbox{ is even}\}
\end{equation*}
and similarly with even replaced by odd.
\end{proof}

\section{Iwasawa Theory}
Recall that $K$ is an imaginary quadratic field of $\Q$ of discriminant $D_K$ and quadratic character $\epsilon$, $p \geq 5$ be a rational prime such that the elliptic curve $E/\Q$ of conductor $N$ has good ordinary or multiplicative reduction at $p$. Also assume that $(D_K, pN)=1$ and $N^-$ be the largest divisor of $N$ such that $(N^-,p)=1$ and $\epsilon(q)=1$ for all primes $q|N^-$. 
We write $N^+ = N/N^-$, so that  $N= N^+N^-$.

As before, let $K^{ac}$ be the anticyclotomic $\Z_p$ extension of $K$ such that the Galois group $\Gamma^- :=\gal(K^{ac}/K) \cong \Z_p$ is characterized by $\tau\sigma\tau={\sigma}^{-1}$ for all $\sigma \in \Gamma^-$ and $\tau$ be a fixed complex conjugation. For any $m\geq0$, consider the subfields $K \subset K_m \subset K^{ac}$ be such that $[K_m:K]=p^m$ and $\Lambda=\Z_p[[\Gamma^-]]$ be the Iwasawa algebra of $\Gamma^-$ over $\Z_p$.

Throughout this section we assume the following hypothesis.
\begin{hypothesis}\label{hyp-N-minus}
\begin{enumerate}
    \item $N^-$ is squarefree,
    \item $(E[p])^{G_K}=0$.
\end{enumerate}
\end{hypothesis}
\subsection{Selmer Modules over $\Lambda$}
\begin{definition}
A degree two prime $\fl \nmid N$ of $K$ is called $k$-admissible if \begin{equation*}
\begin{split}
    N(\fl) &\neq 1 \pmod{p}\\
    E[p^k] &\cong \Z/{p^k}\Z \oplus {\mu}_{p^k} \quad \mbox{ as } \quad \gal(K^{un}_{\fl}/K_\fl)-\mbox{module}.
\end{split}    
\end{equation*}
Let $\cL_k$ denote the set of $k$-\emph{admissible} primes, and $\mathcal{N}_k$ denote the set of square free product of primes in $\mathcal{L}_k$. The set of 1-admissible primes of $\cL$ are simply referred to as \emph{admissible}.
\end{definition}
 Let $T_p$ denote the Tate module of the elliptic curve $E$.
 Since $N^-$ is squarefree, for any rational prime $q$ with $q \mid N^-$, $E$ has multiplicative reduction at $q$, and hence split multiplicative reduction at the prime $\fq$ of $K$ above $q$. By the Tate parametrization, the $G_{K_\fq}$-representation on $T_p$ 
 gives a filtration of $G_{K_{\fq}}$-modules
 \begin{equation*}
 0\subset\operatorname{Fil}_q(T_p) \subset T_p    
 \end{equation*}
 where $\operatorname{Fil}_q(T_p)\cong\Z_p(\epsilon_{cyc})$ and $\epsilon_{cyc}$ is the cyclotomic character of $G_{K_\fq}$.
  For any extension $F/K_{\fq}$, we define the \emph{ordinary} submodule by
 \begin{equation*}
 H^1_{ord}(F,T_p):=\img \left(H^1(F,\operatorname{Fil}_q(T_p)\lra H^1(F,T_p)\right).
 \end{equation*}
 Similarly, we also define $H^1_{ord}(F,E[p^k])$.
 For a $k$-admissible prime $\fl \in \mathcal{L}_k$, we have a similar local condition $H^1_{ord}(F,E[p^k])$ for any extension $F/K_{\fl}$.  
  Since $T_p$ is ordinary, we have a filtration of $G_{\Q_p}$-modules:
 \begin{equation*}
     0\subset\operatorname{Fil}_p(T_p)\subset T_p
 \end{equation*}
 on which the inertia subgroup at $p$ acts on $\operatorname{Fil}_p(T_p)$ by the cyclotomic character.
 We define the ordinary local condition at $p$ by 
 \begin{equation*}
     H^1_{ord}(F,T_p):=
     \img\left(H^1(F,\operatorname{Fil}_p(T_p))\lra H^1(F,T_p)\right)
 \end{equation*}
 for any finite extension $F/{\Q}_p$. 
 
 Replacing $T_p$ by $E[p^k]$ and $E[p^{\infty}]$, we can define the ordinary cohomology groups for these modules. 
 We similarly define ordinary cohomology and selmer groups for the
twists $T_p\otimes\alpha$ for characters $\alpha$ of $\Gamma^-$. Writing $T$ for $T_p\otimes\alpha$, and for every integer $m\geq0$, we define:
\begin{eqnarray*}
    H^1_{ord}(K_{m, \fl},T) = \bigoplus_{w\mid\fl} H^1_{ord}(K_{m,w},T)\\
    H^1_{unr}(K_{m, \fl},T) = \bigoplus_{w\mid\fl} H^1_{unr}(K_{m,w},T).
\end{eqnarray*}
 By \cite[Lemma 3.1.2]{hbes}, 
 for any $k$-admissible prime $\fl \in \mathcal{L}_k$, the modules 
      $\varinjlim_{m} H^1_{ord}(K_{m, \fl},E[p^k])$ and 
      $\varprojlim_{m} H^1_{unr}(K_{m, \fl},E[p^k])$ 
 are free $\Lambda/p^k{\Lambda}$-module of rank one. 
  \begin{definition}[Selmer groups]\label{compact-selmer} We define the compact selmer group by
 \begin{equation*}
 \begin{split}
     \cS(K^{ac},T_p)=\varprojlim_m \ker\left[H^1(K_m,T_p)\lra
     \left(\bigoplus_{\fl\nmid pN^-} \dfrac{H^1(K_{m,\fl},T_p)}{H^1_{unr}(K_{m,\fl},T_p)}
     \bigoplus_{\fl\mid pN^-}\dfrac{H^1(K_{m,\fl},T_p)}{H^1_{ord}(K_{m,\fl},T_p)}\right)\right]
\end{split}
 \end{equation*}
 We write $\cS:=\cS(K^{ac},T_p)$. We similarly define $\cS(K^{ac},E[p^k])$ by replacing $T_p$ by $E[p^k]$. We define
 the standard Selmer group over the anticyclotomic $\Z_p$-extension:
  \begin{equation*}
 \begin{split}
     \Sel(K^{ac},E[p^\infty])=\varinjlim_m\ker\left[ H^1(K_m,E[p^\infty])
     \lra
     \left(\bigoplus_{\fl\nmid pN^-} \dfrac{H^1(K_{m,\fl},E[p^\infty])}{H^1_{unr}(K_{m,\fl},E[p^\infty])}
     \bigoplus_{\fl\mid pN^-}\dfrac{H^1(K_{m,\fl},E[p^\infty])}{H^1_{ord}(K_{m,\fl},E[p^\infty])}\right)\right]
\end{split}
 \end{equation*}
 We also write 
 $X:= \operatorname{Hom}(\Sel(K^{ac}, E[p^{\infty}]),\Q_p/\Z_p)$
 and for any $\nn \in \mathcal{N}_k$:
 \begin{equation*}
     \cS_{\nn}(K^{ac},E[p^k]) \subset \varprojlim_{m} H^1(K_m,E[p^k])
 \end{equation*}
be the $\Lambda$-submodule of classes which are ordinary at the primes dividing ${\nn}p{N^-}$ and unramified at all other primes. Then $\cS=\cS_{1}(K^{ac},E[p^k])$.
\end{definition}
Let $V_p= T_p \otimes \Q_p$. Then we have an exact sequence
\begin{equation*}
    0{\lra}T_p{\lra}V_p{\lra}E[p^{\infty}]{\lra}0
\end{equation*}
Consider prime ideal $\fP \neq p\Lambda$ of height one in $\Lambda$, and let $\OO_{\fP}$  denote the integral closure of $\Lambda/\fP$, viewed as a Galois module with trivial action. 
Then the quotient field $L_\fP$ of $\OO_\fP$ is a finite extension of $\Q_p$. Let $\mm_{\fP}$ denote the maximal ideal of $\OO_{\fP}$.
Let $T_{\fP}=T_p\otimes\Lambda$, viewed as a $G_K$-module via the natural map $G_K{\lra}{\Lambda}^{\times}$ which is non-trivial.  Tensoring with $\OO_{\fP}$, we obtain an exact sequence of $\OO_{\fP}[[G_K]]$-modules
\begin{equation*}
    0{\lra}T_{\fP}{\lra}V_{\fP}{\lra}W_{\fP}{\lra}0. 
\end{equation*}
Now for any prime $\fl$ of $K$ and $M$ be any one of $T_{\fP}, V_{\fP}, 
\mbox{ or } W_{\fP}$, we define the submodule 
\begin{equation*}
    H^1_{\mathcal{F}_{\fP}}(K_{\fl},M) \subset H^1(K_{\fl},M)
\end{equation*}
as follows. First suppose $M=V_{\fP}$. 
We define 
\begin{equation*}
H^1_{\mathcal{F}_{\fP}}(K_{\fl},V_{\fP})=\begin{cases}
H^1_{unr}(K_{\fl},V_{\fP}) & \mbox{ if } \fl\nmid pN^-\\
\img\left[H^1(K_{\fl},\operatorname{Fil}_q(T_p) \otimes L_\fP){\lra} H^1(K_{\fl},V_{\fP})\right] &\mbox{ if } \fl\mid pN^-,
\end{cases}
\end{equation*}
If $M=T_{\fP}$ or $W_{\fP}$ then  $H^1_{\mathcal{F}_{\fP}}(K_{\fl},M)$ is obtained from  $H^1_{\mathcal{F}_{\fP}}(K_{\fl},V_{\fP})$ ( see Remark \ref{Note}). These local submodules define global Selmer groups which we denote by $\Sel_{\mathcal{F}_{\fP}}(K,M)$.
\subsection{Euler Systems  over $\Lambda$}
\begin{definition}
For $\nn \in \mathcal{N}_1$, let $n\in\N$ such that $n{\OO}_K= \nn$. Then $\nn$ is said to be \emph{definite} if $\epsilon(nN^-)=-1$ and \emph{indefinite} if $\epsilon(nN^-)=1$. Let $\mathcal{N}^{definite}_k \subset \mathcal{N}_k$ denote the subset which consists of products of definite primes and $\mathcal{N}^{indefinite}_k$ the set of products of indefinite primes.

We also suppose that for any positive integer $k$ we are given a pair of families:
\begin{equation}\label{ES}
    \{\kappa_{\nn} \in \cS_{\nn}(K^{ac},E[p^k]) | \nn \in \mathcal{N}^{indefinite}_k \} \quad\mbox{ and }\quad
    \{\lambda_{\nn} \in \Lambda/p^k{\Lambda} | \nn \in \mathcal{N}^{definite}_k \}
\end{equation}
which are compatible with the maps $\Lambda/p^{k+1}{\Lambda} \to \Lambda/p^k{\Lambda}$, $E[p^{k+1}] \to E[p^k]$ and the inclusion $\mathcal{N}_{k+1} \subset \mathcal{N}_k$ as $k$ varies. We also assume that it satisfies the first and second reciprocity laws:
\begin{enumerate}
\item for any $\nn\fl \in \mathcal{N}_k^{indefinite}$ there is an isomorphism of $\Lambda$-modules
\begin{equation}\label{rec1}
    \lim_{\stackrel{\longleftarrow}{m}} H^1_{ord}(K_{m,\fl},E[p^k]) \cong \Lambda/p^k{\Lambda}, \mbox{ with } \img(\loc_{\fl}(\kappa_{\nn\fl}))=\lambda_{\nn},
\end{equation}
\item for any $\nn\fl \in \mathcal{N}^{definite}_k$ there is an isomorphism of $\Lambda$-modules
\begin{equation}\label{rec2}
    \lim_{\stackrel{\longleftarrow}{m}} H^1_{unr}(K_{m,\fl},E[p^k]) \cong \Lambda/p^k{\Lambda}, \mbox{ with } \img(\loc_{\fl}(\kappa_{\nn}))=\lambda_{\nn\fl}. 
\end{equation} 
\end{enumerate}
\end{definition}
Since the empty product lies in $\mathcal{N}_k$ for every $k$, we obtain two distinguished elements 
\begin{equation}\label{distinguished}
\begin{split}
    \lambda^{\infty} \in \Lambda &\quad \mbox{if} \quad \epsilon(N^-)=-1\\
    \kappa^{\infty} \in \cS        &\quad \mbox{if} \quad \epsilon(N^-)=1
\end{split}
\end{equation}
which is defined as the inverse limit of $\lambda_1$ and $\kappa_1$ as $k$ varies. 
\begin{lemma}[{\cite[Lemma 3.2.2]{hbes}}]\label{torsion-free}
The $\Lambda$-module $\cS$ is torsion free.
\end{lemma}
Regarding the rank of $\cS$ and the characteristic power series of the torsion part of $X$, we have the following theorem, which is the main theorem here.
\begin{theorem}\label{Main Theorem}
Let the distinguished elements $\lambda^\infty,\kappa^\infty$ be nonzero,  and $X_{{tor}}$ denote the torsion submodule of $X$. Then, we have the following.
\begin{enumerate}
\item The rank formulas:
 $   \operatorname{rank}_{\Lambda}\cS = \operatorname{rank}_{\Lambda}X = \begin{cases} 0 \quad \mbox{if} \quad \epsilon(N^-)=-1\\
    1 \quad \mbox{if} \quad \epsilon(N^-)=1.
    \end{cases}$
\item For any height one prime $\fP$ of $\Lambda$, we have
\begin{equation}\label{main-inequality}
    \ord_{\fP}(\operatorname{char}(X_{{tor}})) \leq 2. 
    \begin{cases} 
    \ord_{\fP}(\lambda^{\infty}) &\quad \mbox{if} \quad \epsilon(N^-)=-1\\
    \ord_{\fP}(\operatorname{char}(\cS/\Lambda{\kappa}^{\infty})) &\quad \mbox{if} \quad \epsilon(N^-)=1.
    \end{cases}
\end{equation}
\item If there exists $s\in\N$ such that for all $t \geq s$ the set
\begin{equation*}
    \{\lambda_{\nn} \in \Lambda/p^t{\Lambda} | \nn \in \mathcal{N}^{definite}_t \}
\end{equation*}
contains an element which is nonzero in $\Lambda/(\fP, p^s)$, then equality holds in the inequality \eqref{main-inequality}.
\end{enumerate}
\end{theorem}
After some preliminary results, we prove this theorem below in Section \ref{main proof}.

\begin{prop}[{\cite[Prop 3.3.2]{hbes}}]\label{Shap}
Shapiro's lemma, the natural map $T_p{\otimes}\Lambda{\lra}T_{\fP}$, and its dual induce maps
\begin{equation*}
    \cS/{\fP}\cS \lra \Sel_{\mathcal{F}_{\fP}}(K,T_{\fP}), \quad
    \Sel_{\mathcal{F}_{\fP}}(K,W_{\fP}){\lra} \Sel(K^{ac},E[p^{\infty}])[\fP]
\end{equation*}
The first map is injective. There is a finite set of height one primes, say $\Sigma_{\Lambda}$ of $\Lambda$, such that if $\fP \notin \Sigma_{\Lambda}$, then these maps have finite kernel and cokernel both of which are bounded by a constant depending on $[\OO_{\fP}: \Lambda/\fP]$ but not on $\fP$ itself.
\end{prop}
\begin{proof}

The proof follows as in \cite[Proposition 5.3.13, 5.3.14]{mr}. There the Cartesian property of the selmer structure for $T/\mm T$ is required  to bound the kernel and cokernel of the second map (which is denoted by $\pi^\ast_\fP$ in \emph{loc. cit}). In our situation, when we have $E[p]^{G_K}=0$, the Cartesian property of the Selmer structure for $T/\mm T$ still holds, which we can use to bound the kernel and cokernel ( see the last paragraph of the proof of \cite[Proposition 5.3.14]{mr}).
\end{proof}

\begin{lemma}[{\cite[Lemma 3.7.1]{mr}}]\label{lambda}
Let $S_{\fP}:=\Sel_{\mathcal{F}_{\fP}}(K,T_{\fP})$. Then the natural map
\begin{equation*}
    S_{\fP}/p^kS_{\fP}{\lra} \Sel_{\mathcal{F}_{\fP}}(K,T_{\fP}/p^kT_{\fP})
\end{equation*}
is injective, where the Selmer structure on $T_{\fP}/p^kT_{\fP}$ is induced from the Selmer structure on $T_{\fP}$ (see Remark \ref{Note}).
\end{lemma}
Let $k$, $j$ be positive integers such that $k \leq j$ and $u_k=\min\{\len(M_\nn^{(k)})\mid \nn\in\cN\}$ for Selmer groups over the Artinian ring $\OO_{\fP}/p^k{\OO_{\fP}}$ as defined in \ref{def-core}. Set
\begin{eqnarray*}
    \delta_{\fP}(k,j)&=& \min \{\ind({\pi}^{\varepsilon}\lambda_{\nn},\pi^{u_k}\OO_{\fP}/p^k{\OO_{\fP}})\mid \nn \in \mathcal{N}^{definite}_j\} \\
                    &=& \min \{\ind(\lambda_{\nn}, \pi^{u_k}\OO_{\fP}/p^k{\OO_{\fP}})\mid \nn \in \mathcal{N}^{definite}_j\}+\varepsilon
\end{eqnarray*}
Set $T^{(i)} =T_{\fP}/p^iT_{\fP}$ and $R^{(i)} = \OO_{\fP}/p^i{\OO_{\fP}}$ for any $i\in\N$.
\begin{remark}
Note that the non-trivial anticyclotomic character $\Gamma^-\lra\Z_p^\times$ acts on $T_\fP$, and hence on $W_\fP$.
Therefore the results in Sections \ref{section-key}, \ref{bipartite} and \ref{section-stub} are applicable to Selmer groups for $T^{(j)}$.
\end{remark}
Given the Selmer structure $\mathcal{F}_{\fP}$ on $T_{\fP}$, let $\mathcal{F}$ be the Selmer structure on $T^{(j)}$ given by Remark \ref{Note}. Similarly, we also denote the selmer structure induced from $\mathcal{F}_{\fP}$ to $W_{\fP}$ to $W_{\fP}[p^j]$ by $\cF_\fP$.
Note that $\varprojlim_{m} H^1(K_m,E[p^j])\cong H^1(K,E[p^j]{\otimes}\Lambda)$. Applying the maps
\begin{equation*}
\begin{split}
    \Lambda/p^j{\Lambda}&{\lra}R^{(j)}\\
    \varprojlim_{m} H^1(K_m,E[p^j]) &{\lra}H^1(K,T^{(j)})
\end{split}
\end{equation*}
to the \emph{Euler system} in \eqref{ES}, we get the pair of families
 \begin{equation}\label{NES}
     \{\ol\kappa_{\nn} \in \Sel_{\mathcal{F}({\nn})}(K,T^{(j)})| \nn \in \mathcal{N}^{indefinite}_j\} \quad
     \{\ol\lambda_{\nn}{\in}R^{(j)}|\nn \in \mathcal{N}^{definite}_j\}.
 \end{equation}
 By assumption and Lemma \ref{lambda} we have $\ol\kappa_1$ or $\ol\lambda_1$ is nonzero (where $1 \in \mathcal{N}_j$ is the empty product depending on whether $\epsilon(N^-)=1$ or $-1$).
 
 Fixing a uniformizer of $\OO_{\fP}$ we have an isomorphism $T^{(j)}{\cong}W_{\fP}[p^j]$, which gives us the isomorphism:
 \begin{equation}\label{Sel[p]}
     \Sel_{\mathcal{F}}(K,T^{(j)}){\cong}\Sel_{\mathcal{F_\fP}}(K,W_{\fP}[p^j]){\cong} \Sel_{\mathcal{F}_{\fP}}(K,W_{\fP})[p^j].
 \end{equation}
 where the last isomorphism follows from Lemma \ref{sel[m]}.
 \begin{lemma}
Consider the triple $(T^{(j)},\mathcal{F}, \mathcal{L}_j)$. Then we have $\mathcal{N}_j ={\mathcal{N}}^{odd}_j \bigsqcup {\mathcal{N}}^{even}_j$. This gives us
 \begin{equation*}
 {\mathcal{N}}^{odd}_j={\mathcal{N}}^{indefinite}_j,\quad {\mathcal{N}}^{even}_j={\mathcal{N}}^{definite}_j.
  \end{equation*}
  Furthermore, the families in \eqref{NES} form an \emph{Euler system of odd type} for $(T^{(j)},\mathcal{F},\mathcal{L}_j)$.
 \end{lemma}
 \begin{proof} The proof goes in the same way as mentioned in \cite[lemma 3.3.5]{hbes}.
\end{proof}
As in \cite{hbes}, the Euler system in equation \eqref{NES} for $(T^{(k)},\mathcal{F}, \mathcal{L}_k)$ may not be free, but this can be obtained by shrinking the set of indexing primes $\mathcal{L}_k$. With suitable modifications, the proof of the results of Howard now goes through.
\begin{lemma}
For any $j{\geq}2k$ and $k>\varepsilon_0$ the families 
\begin{equation*}
     \{\ol{\kappa_{\nn}} \in \Sel_{\mathcal{F}({\nn})}(K,T^{(k)})|\nn \in \mathcal{N}^{indefinite}_j\} \quad\mbox{and}\quad
     \{\ol{\lambda_{\nn}} \in R^{(k)}|\nn \in \mathcal{N}^{definite}_j\}
 \end{equation*}
 form a free Euler system of odd type for $(T^{(k)},\mathcal{F}, \mathcal{L}_j)$.
\end{lemma}
\begin{proof}
The proof of this result is similar to that of \cite[3.3.6]{hbes}.
Let $\nn \in \mathcal{N}^{indefinite}_j$ and $j \geq 2k$. 
By Proposition \ref{selmer-facts}\ref{structure}, we have 
\begin{equation*}
    \Sel_{\mathcal{F}(\nn)}(K,T^{(j)}) \cong R^{(j)}{\oplus}M\oplus M \quad\mbox{and }
     \Sel_{\mathcal{F}(\nn)}(K,T^{(k)}) \cong R^{(k)} \oplus N\oplus N.
\end{equation*}
For the maximal ideal $\mm$ of $\OO_{\fP}$, we fix a uniformizer $\pi$. Let $e$ be the ramification degree of $\OO_{\fP}$. Then the length of $R^{(k)}$ is $ek$. If ${\mm}^{ek-1}N \neq 0$ then we have $\ol{\kappa_{\nn}}=0$ by Proposition \ref{annihilate}, so there is nothing to prove. Suppose that ${\mm}^{ek-1}N=0$. By Lemma \ref{sel[m]} we have  the commutative diagram
\begin{equation*}
  \begin{tikzcd}
    \Sel_{\mathcal{F}({\nn})}(K,T^{(j)}) \arrow{r}{{\pi}^{e(j-k)}} \arrow{d}   & \Sel_{\mathcal{F}({\nn})}(K,T^{(j)})[{\mm}^{ek}] \\
    \Sel_{\mathcal{F}({\nn})}(K,T^{(k)}) \arrow{ur}{\cong}&
  \end{tikzcd}
\end{equation*}
The diagonal isomorphism shows that ${\mm}^{ek-1}.\Sel_{\mathcal{F}({\nn})}(K,T^{(j)})[{\mm}^{ek}]$ is a cyclic module, and this implies ${\mm}^{ek-1}M=0$. But since $j \geq 2k$ the image of $M$ under the vertical arrow is zero. Therefore the image of the vertical arrow is free of rank one, and also $\ol{\kappa_{\nn}}$ is contained in the image. This completes the proof.
\end{proof}
To prove the main Theorem \ref{Main Theorem}, we first show the proposition below. Note that we get an error term, which does not appear in the irreducible case.
 \begin{prop}\label{Sel length}
Let $\epsilon(N^-)=-1$ and assume that $\lambda^{\infty} \in \Lambda$ has nontrivial image in $\OO_{\fP}/p^k{\OO_{\fP}}$. Then 
\begin{equation*}
    \len_{\OO_{\fP}}(\Sel_{\mathcal{F}_{\fP}}(K,W_{\fP}))+2\delta_{\fP}(k+\varepsilon)
    =2.\len_{\OO_{\fP}}(\OO_{\fP}/{\OO_{\fP}{\lambda}^{\infty}})+2\varepsilon.
\end{equation*}
Let $\epsilon(N^-)=1$ and assume that $\kappa^{\infty}$ has nontrivial image in $S_{\fP}/p^kS_{\fP}$. Then 
\begin{enumerate}
    \item $S_{\fP}$ is a free $\OO_{\fP}$-module of rank one,
    \item $\Sel_{\mathcal{F}_{\fP}}(K,W_{\fP})$ has $\OO_{\fP}$-corank one, and
    \item $\len_{\OO_{\fP}}(\Sel_{\mathcal{F}_{\fP}}(K,W_{\fP})/
    {\Sel_{\mathcal{F}_{\fP}}(K,W_{\fP})_{\operatorname{div}}})+2\delta_{\fP}(k+\varepsilon)
    =2.\len_{\OO_{\fP}}(S_{\fP}/\OO_{\fP}{\kappa}^{\infty})+2\varepsilon$, where the subscript $\operatorname{div}$ indicates the the maximal $\OO_{\fP}$-divisible submodule of $\Sel_{\mathcal{F}_{\fP}}(K,W_{\fP})$.
\end{enumerate}
\end{prop}
\begin{proof}
Let us fix $j \geq 2(k+\varepsilon)$. The empty product lies in $\mathcal{N}^{definite}_j$ if and only if $\epsilon(N^-)=-1$, and in which case, $\Sel_{\mathcal{F}}(K,T^{(k+\varepsilon)})\cong M\oplus M$. Since image of $\lambda^\infty$ is non-trivial in $\OO_
\fP/p^k\OO_\fP$, so $\pi^\varepsilon\lambda^\infty$ has non-trivial image in $\OO_\fP/p^{k+\varepsilon}\OO_\fP$. By Theorem \ref{Rigidity} with $\nn =1$, we have $\langle\pi^\varepsilon\ol\lambda_\nn\rangle=\mm^{\delta_\fP(k+\varepsilon,j)+\len_{\OO_{\fP}}(M)}R^{(k+\varepsilon)}$. Therefore,
\begin{eqnarray*}
    \len_{\OO_{\fP}}(M)+ \delta_{\fP}(k+\varepsilon,j)= \ind({\pi}^{\varepsilon}\ol{\lambda_1},R^{(k+\varepsilon)})=& \ind({\pi}^{\varepsilon}{\lambda}^{\infty}, \OO_{\fP}/p^{k+\varepsilon}{\OO_{\fP}})\\
    =& \ind({\lambda}^{\infty}, \OO_{\fP}/p^{k+\varepsilon}{\OO_{\fP}})+\varepsilon.
\end{eqnarray*}
Since the right hand side is $\leq k-1$, by equation \eqref{Sel[p]} and Proposition \ref{annihilate} we have
\begin{equation*}
     \Sel_{\mathcal{F}_\fP}(K,T^{(k)}){\cong}\Sel_{\mathcal{F}_\fP}(K,W_{\fP}[p^{k}]){\cong} \Sel_{\mathcal{F}_{\fP}}(K,W_{\fP})[p^{k}]=\Sel_{\mathcal{F}_{\fP}}(K,W_{\fP})[p^{k-1}]
\end{equation*}
Repeatedly applying Proposition \ref{annihilate} and equation \eqref{Sel[p]}, we have
\begin{equation*}
    \Sel_{\mathcal{F}_{\fP}}(K,W_{\fP})[p^{k}]=\Sel_{\mathcal{F}_{\fP}}(K,W_{\fP})[p^{k-1}]=...=\Sel_{\mathcal{F}_{\fP}}(K,W_{\fP})[p^{\varepsilon_0+1}],
\end{equation*}
since any $x\in \Sel_{\mathcal{F}_{\fP}}(K,W_{\fP})$ is killed by $p^{k-1}$ for $k>\varepsilon_0$. 
As $\Sel_{\cF_\fP}(K,W_{\fP})=\cup_{k>\varepsilon_0}\Sel_{\cF_\fP}(K,W_{\fP})[p^k]$, it follows that $\Sel_{\mathcal{F}_{\fP}}(K,W_{\fP})$ is annihilated by a power of $p$, and 
$\len(\Sel_{\mathcal{F}_{\fP}}(K,W_{\fP}))= \len(M \oplus M)$. Thus,
\begin{equation*}
    \len_{\OO_{\fP}}(\Sel_{\mathcal{F}_{\fP}}(K,W_{\fP}))+ 2.{\delta}_{\fP}(k+\varepsilon,j)= 2.\len_{\OO_{\fP}}(\OO_{\fP}/\OO_{\fP}{\lambda}^{\infty})+2\varepsilon.
\end{equation*}
Now let $\epsilon(N^-)=1$. Then $\Sel_{\mathcal{F}}(K,T^{(k+\varepsilon)})\cong R^{(k+\varepsilon)} \oplus M\oplus M$, applying Theorem \ref{Rigidity} we have
\begin{equation*}
    \len_{\OO_{\fP}}(M)+ \delta_{\fP}(k+\varepsilon,j)= \ind(\pi^\varepsilon\ol{\kappa_1},\Sel_{\mathcal{F}}(K,T^{(k+\varepsilon)})) \leq k-\varepsilon.
\end{equation*}
Again as above we have $\len_{\OO_{\fP}}(M) \leq k-1$. 
Combining this with equation \eqref{Sel[p]}, we have
\begin{equation*}
    S_{\fP}\cong\varprojlim_{k} \Sel_{\mathcal{F}_{\fP}}(K,W_{\fP})[p^k]
\end{equation*}
which is a torsion-free rank one $\OO_{\fP}$-module. Applying Lemma \ref{lambda} we know that the reduction map
\begin{equation*}
    S_{\fP}/p^kS_{\fP} {\lra} \Sel_{\mathcal{F}}(K,T^{(k)})
\end{equation*}
is injective. Theorem \ref{Rigidity} now gives
\begin{equation*}
\begin{split}
    \len_{\OO_{\fP}}(\Sel_{\mathcal{F}_{\fP}}(K,W_{\fP})/\Sel_{\mathcal{F}_{\fP}}(K,W_{\fP})_{\operatorname{div}})
    + 2.{\delta}_{\fP}(k+\varepsilon,j)
                               &= \len_{\OO_{\fP}}(M \oplus M)+ 2{\delta}_{\fP}(k+\varepsilon,j)\\
    = 2.\ind(\ol{\kappa_1}, \Sel_{\mathcal{F_\fP}}(K,T^{(k+\varepsilon)}))+2\varepsilon
                               &= 2.\len_{\OO_{\fP}}(S_{\fP}/S_{\fP}{\kappa}^{\infty})+2\varepsilon.
    \end{split}
\end{equation*}
Now taking $j \to \infty$ we get the result. This proves the proposition.
\end{proof}

\subsection{Proof of Theorem \ref{Main Theorem}}\label{main proof}
Let $\lambda^{\infty}$ or $\kappa^{\infty}$ be nonzero,
accordingly as $\epsilon(N^-)=-1$ or $\epsilon(N^-)=1$. By
Lemma \ref{torsion-free}, $\cS$ is finitely generated
torsion-free $\Lambda$-module. If $\epsilon(N^-)=1$ then the image of
$\kappa^{\infty}$ in $\cS/{\fP}\cS$ is nonzero for all but finitely many
height one prime ideals $\fP$ in $\Lambda$. Similar arguments hold for
$\lambda^{\infty}$ if $\epsilon(N^-)=-1$. So let $\Sigma_{\Lambda}$ be a
finite set of prime ideals of $\Lambda$ large enough that it contains
$p{\Lambda}$ and all the prime divisors of the characteristic ideal of
the torsion submodule $X$, and large enough such that the image of the
distinguished elements in \eqref{distinguished} has nonzero image in
$\cS/{\fP}\cS$ or $\Lambda/{\fP}{\Lambda}$ for all $\fP \notin
\Sigma_{\Lambda}$.

(i) We fix $\fP \notin \Sigma_\Lambda$ and suppose $\epsilon(N^-)=1$. In this case, it follows from Proposition \ref{lambda} that $\kappa^{\infty}$ has nonzero image in $\Sel_{\mathcal{F}_{\fP}}(K,T_{\fP})$. By Proposition \ref{Sel length}, the rank of  $\Sel_{\mathcal{F}_{\fP}}(K,T_{\fP})$ is one and the corank of  $\Sel_{\mathcal{F}_{\fP}}(K,W_{\fP})$ is also one as $\OO_\fP$-modules. Since $\cS$ is torsion-free, it follows from Proposition \ref{Shap}, that
\begin{equation*}
    \operatorname{rank}_{\Lambda}\cS =\operatorname{rank}_{\OO_{\fP}}(\cS \otimes\OO_{\fP})=1.
\end{equation*}
For $X$, by Proposition \ref{Shap}, the map $\Sel_{\mathcal{F}_{\fP}}(K,W_{\fP})\lra\Sel(K^{ac},E[p^\infty])[\fP]$ has finite kernel and cokernel. Taking Pontryagin dual, it follows that $\operatorname{rank}_\Lambda X=1$.

Similarly, if $\epsilon(N^-)=-1$, then by the previous proposition, $\Sel_{\mathcal{F}_{\fP}}(K,W_{\fP})$ is annihilated by a fixed power of $p$. Therefore $\Sel_{\mathcal{F}_{\fP}}(K,T^{(j)})$ is also
annihilated by a fixed power of $p$ for any $j$. As $\Sel_{\cF_\fP}(K,T_\fP)=\varprojlim S_\fP/p^k S_\fP$, it follows from Proposition \ref{Shap}, that $\cS/\fP\cS$ is annihilated by a fixed power of $p$. Lemma \ref{torsion-free},  then shows that $\cS$ is $\Lambda$-torsion. A similar argument shows that $\operatorname{rank}_\Lambda X=0$. This completes the proof of (i).

(ii) Let $\fP$ be a height one prime ideal of $\Lambda$ which is generated by a distinguished polynomial $g$. 
For each positive integer $m$, consider the height one prime ideals
\begin{equation*}
    \fP_m=\begin{cases}
    (g+p^m)\Lambda &\mbox{ if } g\Lambda\neq p\Lambda\\
    ((\gamma-1)^m+p)\Lambda &\mbox{ if } g\Lambda=p\Lambda
    \end{cases}
\end{equation*}
It is clear that for large enough $m$, $\fP_m$ is a prime ideal which does not lie in $\Sigma_{\Lambda}$. By Hensel's lemma we have $\Lambda/\fP_m \cong \Lambda/\fP$. Then, as in the proof of \cite[Theorem 5.3.10]{mr} and using Proposition \ref{Shap}, we have 
\begin{equation*}
    \len_{\Z_p}(\Sel_{\mathcal{F}_{\fP_{\mm}}}(K,W_{{\fP}_{\mm}})/\Sel_{\mathcal{F}_{\fP_{\mm}}}(K,W_{{\fP}_{\mm}})_{\operatorname{div}})=m.\operatorname{rank}_{\Z_p}(\OO_{\fP}).\ord_{\fP}(\operatorname{char}(X_{\Lambda_{tor}}))
\end{equation*}
up to $O(1)$ as $m$ varies. Let $\cS_{\fP_m}=\Sel_{\mathcal{F}_{\fP_m}}(K,T_{\fP_{m}})$. Then 
\begin{equation*}
\begin{split}
    \len_{\Z_p}(\cS_{\fP_m}/\OO_{\fP_{m}}{\kappa}^{\infty}) & = m.\operatorname{rank}_{\Z_p}(\OO_{\fP}).\ord_{\fP}(\operatorname{char}(\cS/\Lambda{\kappa}^{\infty}))\\
    \len_{\Z_p}(\OO_{\fP_m}/\OO_{\fP_{m}}{\lambda}^{\infty}) &= m.\operatorname{rank}_{\Z_p}(\OO_{\fP}).\ord_{\fP}({\lambda}^{\infty})
    \end{split}
\end{equation*}
accordingly as $\epsilon(N^-)=1 \mbox{ or } -1$, up to $\OO(1)$ as $m$ varies. 
Let $e$ be the absolute ramification degree of $\OO_{\fP_m}$. This ramification degree is independent of $m$. Now for large enough $k$ by Proposition \ref{Sel length} we have
\begin{equation*}
\begin{split}
    \len_{\OO_{\fP_m}}(\Sel_{\mathcal{F}_{\fP_m}}(K,W_{\fP_m})/\Sel_{\mathcal{F}_{\fP_{\mm}}}(K,W_{{\fP}_{\mm}})_{\operatorname{div}}) & + 2e.{\delta}_{\fP_m}(k+\varepsilon)\\
    &=2.\len_{\OO_{\fP_m}}(\cS_{\fP_m}/\OO_{\fP_m}{\kappa}^{\infty})+2\varepsilon,\\
    \len_{\OO_{\fP_m}}(\Sel_{\mathcal{F}_{\fP_m}}(K,W_{\fP_m}))+ 2e.{\delta}_{\fP_m}(k+\varepsilon)&= 2.\len_{\OO_{\fP_m}}(\OO_{\fP_m}/\OO_{\fP_m}{\lambda}^{\infty})+2\varepsilon
\end{split}
\end{equation*}
accordingly as $\epsilon(N^-)=1 \mbox{ or } -1$. Since $\delta_{\fP_m}(k+\varepsilon) \geq 0$, so taking limit 
$m \to \infty$ we get (ii).

(iii) We show that ${\delta}_{\fP_m}(k)$ is bounded as $m,k$ vary for $j \geq k_0>\varepsilon$. 
Let $n(j) \in \mathcal{N}^{definite}_j$ be such that ${\lambda}_{n(j)}$ has nonzero image in $\Lambda/(\fP_m, p^{k_0})$. Then ${\lambda}_{n(j)}$ has nontrivial image in $\Lambda/(\fP_m, p^{k_0})$ for all $m \geq k_0$.
Define
\begin{equation*}
    C_m=\operatorname{coker}[\Lambda/\fP_m \hookrightarrow \OO_{\fP_m}].
\end{equation*}
Note that $C_m$ are finite and up to isomorphism do not depend on $m$. If $k_1$ is large enough that $p^{k_1-k_0}$ kills $C_m$, then we have the following commutative diagram
\begin{equation*}
\xymatrix{
      C_m[p^{k_1}] \ar[r] \ar[d] & \Lambda/(\fP_{\mm},p^{k_1}) \ar[r] \ar[d] & \OO_{\fP_{\mm}}/p^{k_1}{\OO_{\fP_{\mm}}} \ar[d] \\
      C_m[p^{k_0}] \ar[r] & \Lambda/(\fP_{\mm},p^{k_0}) \ar[r] & \OO_{\fP_{\mm}}/p^{k_0}{\OO_{\fP_{\mm}}}.
}    
\end{equation*}
It follows that $\lambda_{n(j)}$ has nontrivial image in $\OO_{\fP_m}/p^{k_1}{\OO_{\fP_m}}$, so 
$\pi^\varepsilon\lambda_{n(j)}$ has non-trivial image in
$\OO_{\fP_m}/p^{k_1+\varepsilon}{\OO_{\fP_m}}$. By Theorem \ref{Rigidity}, $\pi^\varepsilon\lambda_{n(j)}\in\mm^{u_{k_1+\varepsilon}}\OO_{\fP}/p^{k_1+\varepsilon}{\OO_{\fP_m}}$, and by the observation in equation \eqref{surj-ind}, for $j \geq k+\varepsilon \geq k_1+\varepsilon$ we have
\begin{equation*}
    \delta_{\fP_m}(k+\varepsilon,j) \leq \ind({\pi}^{\varepsilon}\lambda_{n(j)},{\mm}^{u_{k+\varepsilon}}\OO_{\fP_m}/p^{k+\varepsilon}{\OO_{\fP_m}}) < 2e(k_1+\varepsilon).
\end{equation*}
This implies $\delta_{\fP_m}(k) < 2e(k_1+\varepsilon)$ for all $k \geq k_1$ and any $m \geq k_0$.
\hfill\qedsymbol

\vspace{1cm}
\begin{minipage}{10cm}
Chandrakant S Aribam\\
Department of Mathematical Sciences,\\
IISER Mohali, \\ 
Sector 81, P.O. Manauli, Punjab 140306\\
Email: aribam@iisermohali.ac.in
\end{minipage}
\hfill
\begin{minipage}{10cm}
Pronay Kumar Karmakar\\
Department of Mathematical Sciences,\\
IISER Mohali, \\
Sector 81, P.O. Manauli, Punjab 140306\\
Email: pronaykarmakar@iisermohali.ac.in
\end{minipage}
\end{document}